\documentclass[a4paper,10pt]{article}
\usepackage{amsfonts,amsthm,amsmath,amssymb,amscd,graphicx,epsfig}
\usepackage{color}
\usepackage{subfigure}
\usepackage{xcolor}

\definecolor{Sepia}{rgb}{0.6,0.3,0.1}
\definecolor{RawSienna}{rgb}{0.2,0.7,0.1}
\definecolor{NavyBlue}{rgb}{0.2,0,0.5}
\definecolor{PineGreen}{rgb}{0,0.5,0}
\definecolor{Sepia1}{rgb}{0.9,0.0,0.1}
\definecolor{Kirpich}{rgb}{0.8,0.1,0.1}
\definecolor{Sepia2}{rgb}{0.4,0.,0.4}
\definecolor{Sepia3}{rgb}{0.4,0.,0.}
\definecolor{Sepia4}{rgb}{0.7,0.,0.7}
\definecolor{Sepia5}{rgb}{0.4,0.,0.7}
\theoremstyle{plain}
\newtheorem{thrm}{Theorem}

\newtheorem{prpstn}{Proposition}
\newtheorem{lmm}{Lemma}

\theoremstyle{definition}

\newtheorem{rmrk}{Remark}

\title{Bifurcations of cubic homoclinic tangencies\\ 
in two-dimensional symplectic maps
~\footnote{
This work has been partially supported by the
Russian Scientific Foundation Grant~14-41-00044, and the Spanish MINECO-FEDER
Grant~MTM2015-65717-P. The numerical experiments in Section~3 have been  carried out 
within the RSciF-grant (project~No.14-12-00811).
Author~MG 
has been partially supported by Juan de la Cierva-Formaci\'on Fellowship~FJCI-2014-21229. 
MG warmly thanks the Department of Mathematics
of Uppsala University for their hospitality and support;  
during her stay at Uppsala University, MG was also partially supported by 
the Knut and Alice Wallenberg Foundation grant~2013-0315.
Author~SG has been partially supported by the grant of RFBR No.~16-01-00324 {and 14-01-00344}.
Author~IO has been supported by the project M2 (Systematic multi-scale modelling and analysis for geophysical flow) of the Collaborative Research Centre TRR 181 ``Energy Transfer in Atmosphere and Ocean'' funded by the German Research Foundation.
}}
\author{M. Gonchenko$^1$, S.V. Gonchenko$^2$ and I. Ovsyannikov$^3$ \\
{\small $^1$ Departament de Mathem\`atiques i Inform\`atica, Universitat de Barcelona, Spain
}\\
{\small $^2$  N.I.~Lobachevsky Nizhny Novgorod
 University, Russia}\\
{\small $^3$  Fachbereich Mathematik und Informatik, Universit\"at Bremen, Germany}\\
{\small { \tt  gonchenko@ub.edu, \tt gonchenko@pochta.ru, \tt ivan.i.ovsyannikov@gmail.com}}
}
\date{}

\begin{document}

\maketitle

\begin{abstract}
We study bifurcations of cubic homoclinic tangencies in two-dimensional symplectic maps.
We distinguish two types of cubic homoclinic tangencies, and each type gives different first return maps
derived to diverse  conservative cubic H\'enon maps with quite different bifurcation diagrams.
In this way, we establish the  structure of bifurcations of periodic orbits in two parameter
general unfoldings generalizing to the conservative case the results previously obtained for the dissipative case.
We also consider the problem of 1:4 resonance for the  conservative cubic
H\'enon maps.
\end{abstract}

\section{Introduction and main results}
Homoclinic tangencies, i.e. tangencies between stable and unstable invariant manifolds of the same saddle periodic orbit,
play a very important role in the theory of dynamical chaos.
Starting with
the works by Smale and Shilnikov~\cite{Sma65,Shi67} (and in the ideological sense as far back as with the
 memoirs by Poincar\'e and Birkhoff~\cite{Poi90, Poi99,Bir35}),  the existence of transversal homoclinic 
intersections is considered as
the universal criterium of the complexity {of} dynamical systems.
At the same time, the  presence of non-transversal homoclinic orbits (homoclinic tangencies) indicates
an extraordinary richness of
bifurcations of such systems and, what is very important, {the principal} impossibility of providing of a
complete description of
bifurcations of such systems within framework of finite parameter families,~\cite{GST93a,GST99,Kal00,DN05,GST07}. Therefore, when
studying homoclinic bifurcations, the main problems are related to the analysis of their principal 
bifurcations and
characteristic properties of dynamics as a whole.

One of the important properties of homoclinic tangencies is that they can freely produce homoclinic and 
heteroclinic
tangencies of higher orders~\cite{GST93a,GST96,GST99,GST07}. In particular, in a two parameter family which 
unfolds generally a
quadratic homoclinic tangency, cubic homoclinic tangencies appear unavoidably.
This means that the study of bifurcations
of cubic homoclinic tangencies  itself becomes an important problem. Moreover, cubic homoclinic tangencies
can play an independent role when studying global bifurcations in many dynamical models. For example, such tangencies
appear naturally in the problem on periodically perturbed two-dimensional flows with a homoclinic 
figure-eight of
a saddle equilibrium~\cite{GSV13}, where the corresponding bifurcation points of codimension two 
give rise to the curves
of quadratic homoclinic tangencies forming (non-smooth) boundaries of homoclinic zones. Moreover, bifurcations
of cubic tangencies {lead to} the appearance of specific windows of stability (i.e. parameter domains corresponding to
the existence of  stable periodic orbits) which are well observable during  numerical explorations,~\cite{GSV13},
unlike the windows produced by quadratic homoclinic tangencies.

In fact, bifurcations of cubic homoclinic tangencies were studied first in~\cite{Gon85} for
the {sectionally} dissipative case (when a saddle periodic orbit 
has {eigenvalues} 
$\lambda_1,\ldots,\lambda_n, \gamma$ such that $|\gamma|>1>|\lambda| = \max\limits_i |\lambda_i|$ 
and $\sigma = |\lambda\gamma|<1$),
see also~\cite{GST96,GSV13, Tatjer}. Note that in~\cite{Gon85} the two main
cases were considered when the leading stable {eigenvalue} is real (i.e. $|\lambda_1|>|\lambda_i|, i=2,\ldots,n,$) 
and
the two leading {eigenvalues} are complex conjugate (i.e. $\lambda_{1,2}= \rho e^{\pm i\varphi}$,
$\varphi \neq 0, \pi$, and  $ \rho > |\lambda_i|, i=3,\ldots,n$).
The main attention in~\cite{Gon85} was focused on
the study of bifurcations of
\emph{single-round periodic orbits}, i.e. periodic orbits which  pass a neighborhood of the homoclinic orbit
only once.

We note that the study of such orbits
can be reduced to  an analysis of fixed points of the corresponding first return maps
defined near some point of the homoclinic tangency. In the case of cubic tangencies, such maps can be brought
to the maps close to the two-dimensional cubic H\'enon maps of the form
\begin{equation}
\mathbf{C}_\pm^{J}: \bar x = y, \;\;\; \bar y= M_1 + M_2 y - J x   \pm y^3,
\label{cubHM}
\end{equation}
being $M_1$ and $M_2$ parameters and $J$ the Jacobian.
Moreover, map~$\mathbf{C}_{-}^J$ appears in the case of cubic tangency 
{incoming from below} (Figure~\ref{fig2type}(a)) and map~$\mathbf{C}_{+}^J$ arises
in the case of cubic tangency
{incoming from above} (Figure~\ref{fig2type}(b)).
In the dissipative case $\sigma <1$, maps~(\ref{cubHM}) with $|J|\ll 1$  were derived 
(using the rescaling method,~\cite{TY86}) in~\cite{GST96}, where the corresponding formulas for the rescaled 
first return maps near homoclinic tangencies of arbitrary finite orders were obtained. 
It is worth mentioning that
maps~$\mathbf{C}_{+}^J$ and~$\mathbf{C}_{-}^J$  were studied in~\cite{GK88} for the
dissipative case~$|J|<1$, where bifurcation diagrams for fixed points and 2-periodic orbits were
constructed analytically
as well as strange attractors
were analyzed numerically.
Also maps~(\ref{cubHM}) are good approximations of the corresponding rescaling first return maps~$T_k$ when the
return time~$k$ is sufficiently
large (in fact,~$k$ is the period of a single-round periodic  orbit under consideration).
Then $(x,y)$ and $(M_1,M_2)$ are rescaled coordinates and parameters, respectively, which are defined in
domains $\|(x,y)\|<D_k$  and $\|(M_1,M_2)\|< P_k$, where $D_k,P_k\to + \infty$ as~$k\to\infty$, the Jacobian
$J$ is proportional to $\sigma^k$ (i.e. $|J|\ll 1$ in the dissipative case and~$J=1$ in the conservative case),  
and the error terms  are of order~$O(\lambda^k)$.

\begin{figure}[bt]
\centering
\includegraphics[height=5cm]{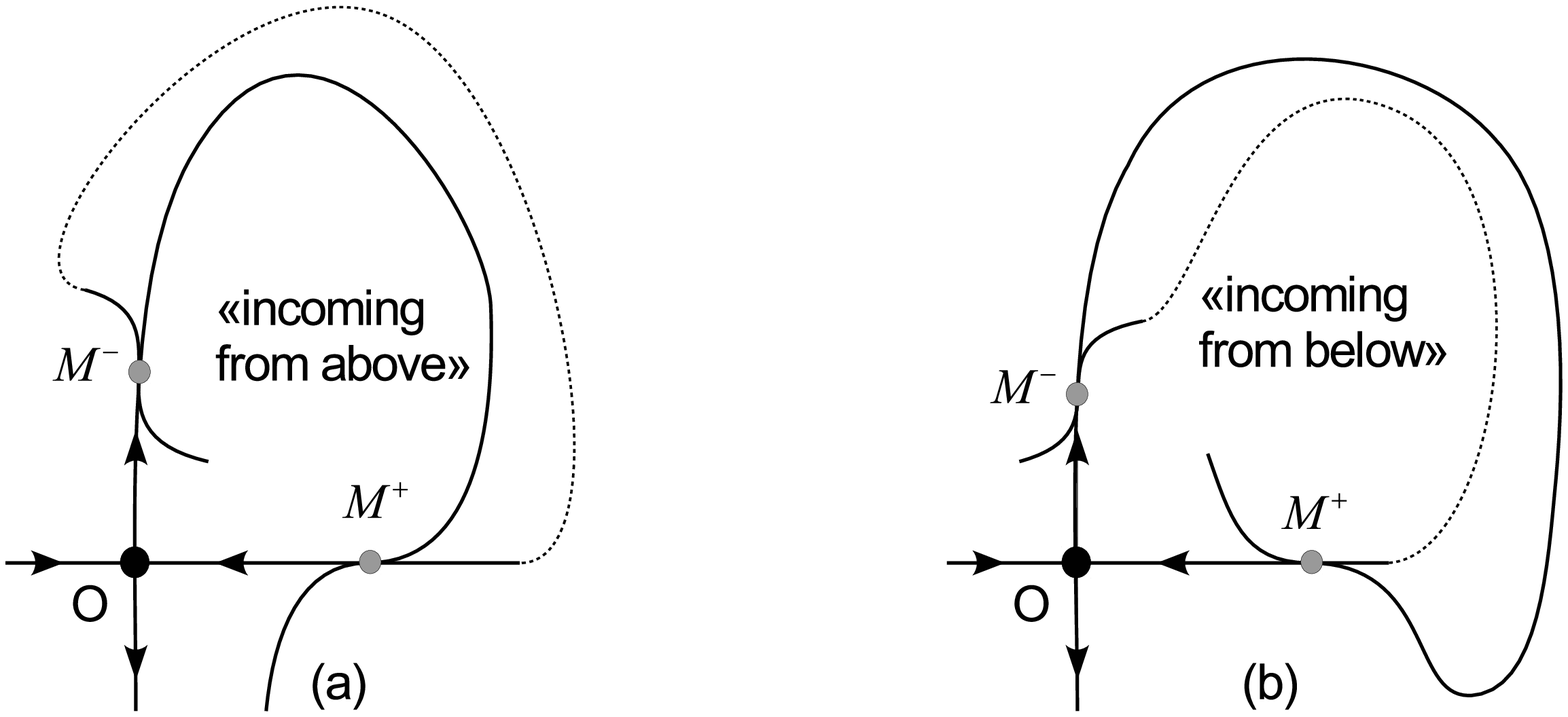}
 \caption{Two types of cubic homoclinic tangencies:
(a) incoming from above and (b) incoming from below.}
\label{fig2type}
\end{figure}

In the present paper we study bifurcations of single-round periodic orbits in the case of two-dimensional symplectic maps
with cubic
homoclinic tangencies. We put emphasis on  bifurcations leading to the birth
of elliptic periodic orbits. The main results can be briefly formulated as follows

\begin{thrm}
Let $f_\mu$, $\mu=(\mu_1,\mu_2)$, be
a two parameter family of symplectic maps  unfolding generally the cubic homoclinic tangency to the saddle fixed point $O$
with {eigenvalues} 
$\lambda$ and $\lambda^{-1}$, where $|\lambda|<1$. Let $V$ be a small neighborhood of the origin in the $(\mu_1,\mu_2)$-plane.
Then the following holds
\begin{enumerate}
\item[\rm 1.] In $V$, there exists
a bifurcation curve {(semicubical parabola)~$H_0$} with the cusp-point $\mu = (0,0)$
such that for $\mu \in H_0 \setminus \{(0,0)\}$, map~$f_\mu$ has a quadratic homoclinic tangency branched from 
the initial cubic tangency.

\item[\rm 2.]
In $V$, there exists {an infinite sequence of {elliptic zones} {$E_k$} accumulated to $H_0$ as $k\to\infty$ such that 
for $\mu \in E_k$,
map~$f_\mu$ has an elliptic single-round
orbit of period $k$.} The boundary of $E_k$ consists of curves $L_k^+$ and $L_k^-$ corresponding to the appearance of
single-round periodic orbits with
double {eigenvalues} $+1$ and $-1$, respectively. 
The elliptic orbit is generic for all values of $\mu\in E_k$, except for those belonging to the strong 1:3 and 1:4 resonance 
curves $L_k^{2\pi/3}$ and $L_k^{\pi/2}$
(the orbit has  {eigenvalues}
$e^{\pm i 2\pi/3}$ and $e^{\pm i\pi/2}$) and the non-twist curve $L_k^{(0)}$ (the first Birkhoff coefficient equals zero).
\end{enumerate}
\label{thm:main}
\end{thrm}

An illustration to Theorem~\ref{thm:main} is shown in Figure~\ref{fig:bdiagr2}. We see that the bifurcation diagrams
for single-round periodic orbits are different in the cases
(a)~$\lambda>0$ and the tangency is {incoming from below}; (b)~$\lambda>0$ and the tangency is {incoming from above}; 
(c)~$\lambda<0$.
Correspondingly, we label zones~$E_k$
as~$E_k^+$  (squids)
in the case (a) and~$E_k^-$ (cockroaches) in the case (b). Notice that in the case $\lambda<0$ zones~$E_k$ of 
different types alternate
depending on the parity of~$k$, and, therefore, for more definiteness, we enumerate these zones as $E_{2m}^+, E_{2m+1}^-,...$.

\begin{figure}[bt]
\centering
\includegraphics[height=7cm]{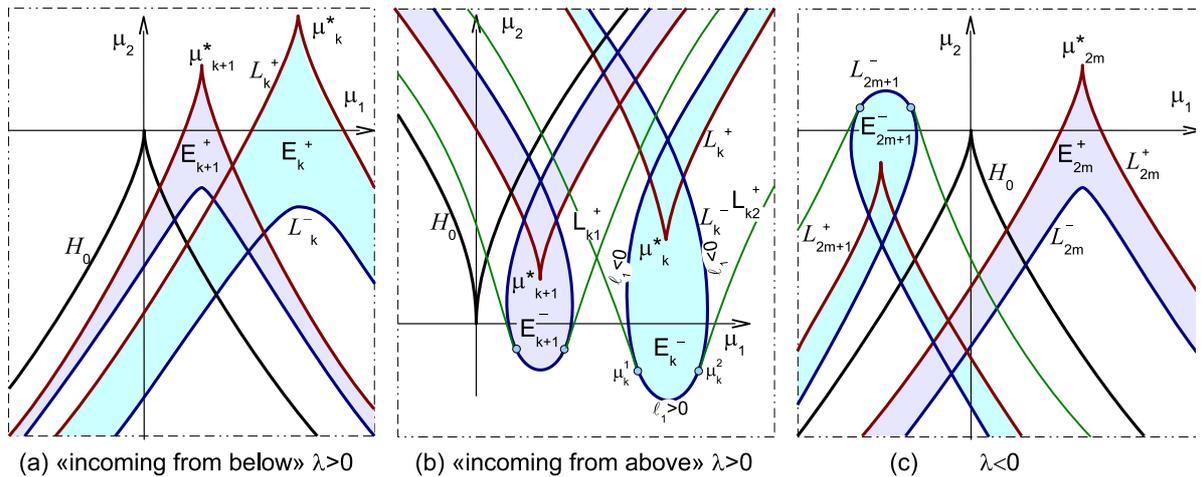}
\caption{{   Main elements of the bifurcation diagram for single-round periodic orbits in the
$(\mu_1,\mu_2)$-plane, where curve~$H_0$
and elliptic zones~$E_k^+$ (a) and~$E_k^{-}$ (b) are shown. These zones alternate for $\lambda<0$ (c). Curves~$L_k^+$ 
correspond to
nondegenerate conservative fold bifurcations except for cusp-points~$\mu_k^*$ of conservative pitch-fork bifurcations.
Curves~$L_k^-$ correspond to nondegenerate conservative period doubling bifurcations except for points~$\mu_k^1$ and~$\mu_k^2$ of
curve~$L_k^-$ of zone~$E_k^-$ where the first Lyapunov value~$\ell_1$ vanishes. 
Points~$\mu_k^1$ and~$\mu_k^2$ are the endpoints
of bifurcation curves~$L_{k1}^+$ and~$L_{k2}^+$ that
correspond to
conservative fold bifurcations of double-round periodic orbits (2-periodic orbits in the first return maps).}}
\label{fig:bdiagr2}
\end{figure}

We prove Theorem~\ref{thm:main} at the end of Section~\ref{sec:stateprob} (item~1) and 
Section~\ref{sect:bifcub} (item~2).
The proof of item~2 of Theorem~\ref{thm:main}
is mainly based on the rescaling results of Lemma~\ref{henmain-c} 
which show that
the corresponding first return map~$T_k$ can be brought, by linear changes of coordinates and parameters,
to a regular map asymptotically close as~$k\to\infty$ to the conservative cubic H\'enon map of the form~(\ref{cubHM}) with~$J=1$.
Moreover, we obtain map~$\mathbf{C}_+^1$
in the case of incoming from below tangency and map~$\mathbf{C}_-^1$ 
in the case of incoming from above tangency
(in the case $\lambda<0$, maps~$\mathbf{C}_+^1$ and~$\mathbf{C}_-^1$
appear, respectively, for even and odd numbers~$k$).
Bifurcations of fixed points and accompanied bifurcations of 2-periodic orbits for these conservative cubic H\'enon maps
are described in Section~\ref{sect:cubicHenon}. The corresponding bifurcation diagrams are shown in Figures~\ref{GK88-+2} 
and~\ref{GK88--2}
for the maps~$\mathbf{C}_+^1$ and~$\mathbf{C}_-^1$,
respectively. We also give the analytical expressions for the corresponding bifurcation curves,
see formulas~(\ref{eq:criv}).
After this,
the proof of Theorem~\ref{thm:main} becomes quite straightforward, see Section~\ref{sec:endofproof}.

Note that in the symplectic case, one of the main questions when studying elliptic points is related to their generic 
properties.
In the case of the conservative H\'enon map $\bar x = y, \bar y = M -x - y^2$, the problem of genericity (stability) 
of elliptic fixed points
was completely solved in~\cite{Bir87}, see also~\cite{SV09}.
Concerning the  conservative cubic H\'enon maps, the problem of genericity (stability) of fixed elliptic points has been
(almost completely) solved in~\cite{DM00}, where conditions for the KAM-stability of these points were obtained. Note also that
bifurcations  of lower-periodic orbits
(of period 1, 2, 3)
were also studied in~\cite{DM00}.

Concerning 4-periodic orbits, in Section~\ref{sec:14rez} we carry out
the corresponding bifurcation analysis {paying} main attention to those phenomena which are connected with 
the peculiarities of bifurcations
of fixed points with {eigenvalues} $e^{\pm i \pi/2}$, i.e. 1:4 resonance. We perform this study separately for
map $\mathbf{C}_+^1$ in Section~\ref{sec:14rez+} and for map $\mathbf{C}_-^1$
in Section~\ref{sec:14rez-}.

Note that the main complex local normal form of a map having a fixed point with {eigenvalues} $e^{\pm i\pi/2}$ can be written as follows,~\cite{Arn-Geom},
\begin{equation}
\begin{array}{l}
\bar \zeta=i(1+\beta)\zeta + i B_1 |\zeta|^2 \zeta +
          i B_{03}\zeta^{*3}+O(|\zeta|^5),
\label{HeComplNew13}
\end{array}
\end{equation}
where $\beta$ is a parameter characterizing deviation of the
angle argument $\varphi$ of {eigenvalues} of the fixed point from~$\pi/2$
($\varphi = \beta +  \pi/2$),
the coefficients $B_1 :=B_1(\beta)$ and $B_{03}:=B_{03}(\beta)$ are real and smoothly depend on~$\beta$.
Then the following conditions,~\cite{Arn-Geom},
\begin{equation}
|B_{03}| \neq 0, \;\;\; A:=\frac{|B_{1}(0)|}{|B_{03}(0)|} \neq 1
\label{D21D03}
\end{equation}
imply that bifurcations of the fixed point in~(\ref{HeComplNew13}) are nondegenerate.

We show that conditions~(\ref{D21D03}) can be violated for both maps~$\mathbf{C}_+^1$ and $\mathbf{C}_-^1$.  Namely, the case~$B_{03} = 0$ occurs for map~$\mathbf{C}_+^1$ and
the equality~$A=1$ can hold for  map~$\mathbf{C}_-^1$. This means that the structure of 1:4 resonance is different in these maps and we study the accompanied bifurcation phenomena in the present paper.

\begin{rmrk} As far as we know, the case $B_{03}=0$, when the fixed point 
in the normal form~(\ref{HeComplNew13}) is a nonlinear center (map~(\ref{HeComplNew13}) without~$O(|z|^5)$ terms is 
a nonlinear rotation) was not considered before. Nevertheless, this case is very 
interesting from various points of view. In particular, the problem of the existence of quasi-central elliptic points 
in reversible analytical maps is very important. Here,  bifurcations of such points (when the map is not conservative) 
lead to the birth of  pairs of periodic orbits sink-source and, as a consequence, to the break-down of 
conservativity and appearance of the so-called mixed dynamics~\cite{GST97,LSt04,DGGLS13} even inside symmetric 
elliptic islands~\cite{GLRT14}.
\label{rem01}
\end{rmrk}

\begin{rmrk} The case~$A=1$ of 1:4 resonance was considered in \cite{Bir87} for the conservative H\'enon map, 
see also~\cite{SV09}. However, for the cubic H\'enon map~$\mathbf{C}_-^1$, the case~$A=1$ is of another 
(more complicated) structure. In particular, proper bifurcations of a fixed point with 
{eigenvalues}~$e^{\pm i\pi/2}$ are replaced here by extrinsic bifurcations of some 4-periodic orbits (thus, 1:4 resonance demonstrates  certain non-twist properties).
\label{rem02}
\end{rmrk}

In the present paper we only study the most interesting and principal peculiarities of 1:4 resonance in maps~$\mathbf{C}_+^1$ and~$\mathbf{C}_-^1$. 
In addition to the analytical considerations, we also display certain numerical 
results in Figures~\ref{pi2plusres} and~\ref{fig_pi2-A1}. We plan to continue to study this in more detail in the nearest future.

The paper is organized as follows. In Section~\ref{sec:stateprob} we start with the statement of the problem and
also  present the results on the normal forms of symplectic saddle maps
(Lemmas~\ref{lem:nfgen} and~\ref{lem:Tknfgen}) and the normal form of global map~$T_1$
defined near a homoclinic point (Lemma~\ref{nfT1}).  We also prove item 1 of Theorem~\ref{thm:main} in
Section~\ref{sec:stateprob}, see Proposition~\ref{th1g85}. In Section~\ref{sect:bifcub} we complete the proof of 
Theorem~\ref{thm:main}. The main technical result there is the Rescaling Lemma, Lemma~\ref{henmain-c}, which shows that 
the first return maps can be written, in some rescaling coordinates, as maps asymptotically close to the conservative 
cubic H\'enon maps~$\mathbf{C}_+^1$ and~$\mathbf{C}_-^1$. Once we know the structure of bifurcations of fixed points 
in these maps we can easily recover bifurcations of single-round periodic orbits and, hence, prove the theorem. 
In Section~\ref{sec:14rez} we consider the problem on the structure of 1:4 resonance in maps~$\mathbf{C}_+^1$ 
and~$\mathbf{C}_-^1$.

\section{Statement of the problem.} \label{sec:stateprob}

Let $f_0$ be a $C^r$-smooth, $r\geq 5$, two-dimensional symplectic
diffeomorphism which satisfies the following conditions:
\begin{enumerate}
\item[{\bf (A)}] $f_0$ has a saddle fixed point~$O$ with
{eigenvalues}~$\lambda$ and~$\lambda^{-1}$, where $|\lambda|<1$;
\item[{\bf (B)}] the invariant manifolds $W^u(O)$ and $W^s(O)$
have a cubic homoclinic tangency at the points of some homoclinic
orbit~$\Gamma_0$.
\end{enumerate}

We distinguish two types of cubic tangencies of~$W^u(O)$ and~$W^s(O)$ at a homoclinic point: the tangency of first type
({incoming from above}) and the tangency of the second type ({incoming from below}). Both these types are shown in  
Figure~\ref{fig2type}.
When the {eigenvalue} $\lambda$ is positive the type of tangency remains the same for all points of the homoclinic orbit.
However, if $\lambda$ is negative the {incoming from above} and {incoming from below} tangencies alternate from point to point.

Let ${\cal H}_2$ be a (codimension two) bifurcation surface
composed of symplectic $C^r$-maps close to~$f_0$ and such that
every map of~${\cal H}_2$ has a
close to $\Gamma_0$
homoclinic orbit at whose points the manifolds $W^u(O)$ and $W^s(O)$ have a cubic tangency.
Let~$f_{\varepsilon}$ be a
family of symplectic
$C^r$-maps that contains map~$f_0$ at $\varepsilon =0$. We
suppose that the family depends smoothly on parameter~$\varepsilon$ and satisfies
the following condition:

\begin{itemize}
\item[{\bf (C)}] the family $f_{\varepsilon}$ is transverse to
${\cal H}_2$ at~$\varepsilon=0$.
\end{itemize}

It is natural to assume that condition~\textbf{C} holds for a two parameter family~$f_\varepsilon$, where
$\varepsilon = (\mu_1,\mu_2)$ and~$\mu_1$ and~$\mu_2$ are some parameters which split generally
the initial cubic tangency of the manifolds $W^u(O)$ and $W^s(O)$ at some homoclinic point.
We consider this question in more detail when proving Lemma~\ref{nfT1} and we deduce that,
without loss of generality, one can assume that~$\varepsilon\in \mathbb{R}^2$.

Let~$U$ be a sufficiently small fixed neighborhood of $O\cup\Gamma_0$. It consists of a small
disk~$U_0$ containing~$O$ and a number of small disks $u_i, i=1,2,\ldots,$ surrounding
those points of~$\Gamma_0$ that do not lie in~$U_0$ (see
Figure~\ref{fg:fabhab}(a)). We call a periodic or homoclinic to $O$ orbit entirely lying in~$U$ to be \emph{$p$-round} if it has exactly~$p$ intersection points with every disk $u_i$. For $p=1$ and $p=2$, we use the terms {\em single-round} and
\emph{double-round} orbits, respectively. 
\begin{figure}[tb]
\centering
    \includegraphics[width=0.4\textwidth]{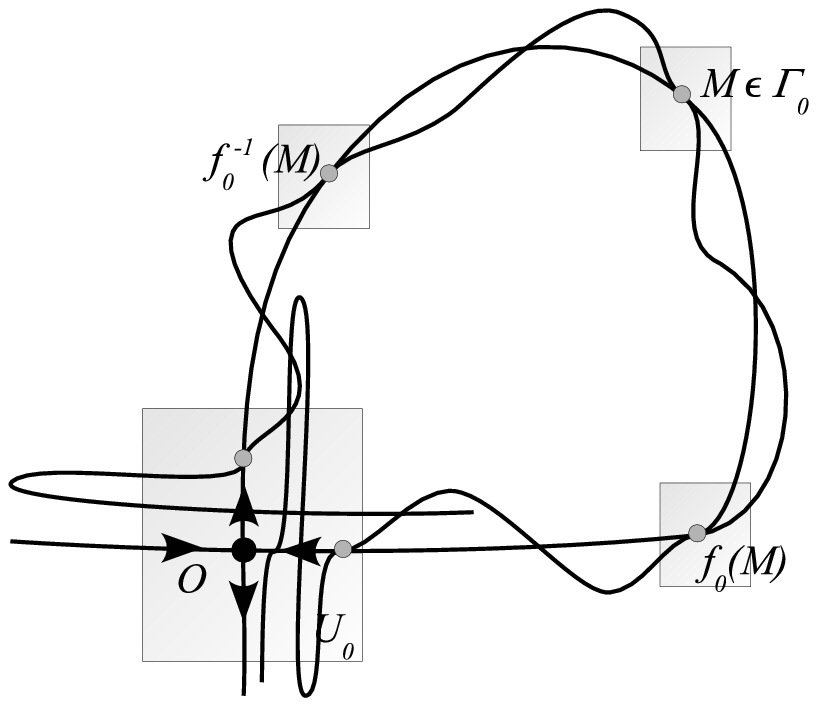}
   \includegraphics[width=0.3\textwidth]{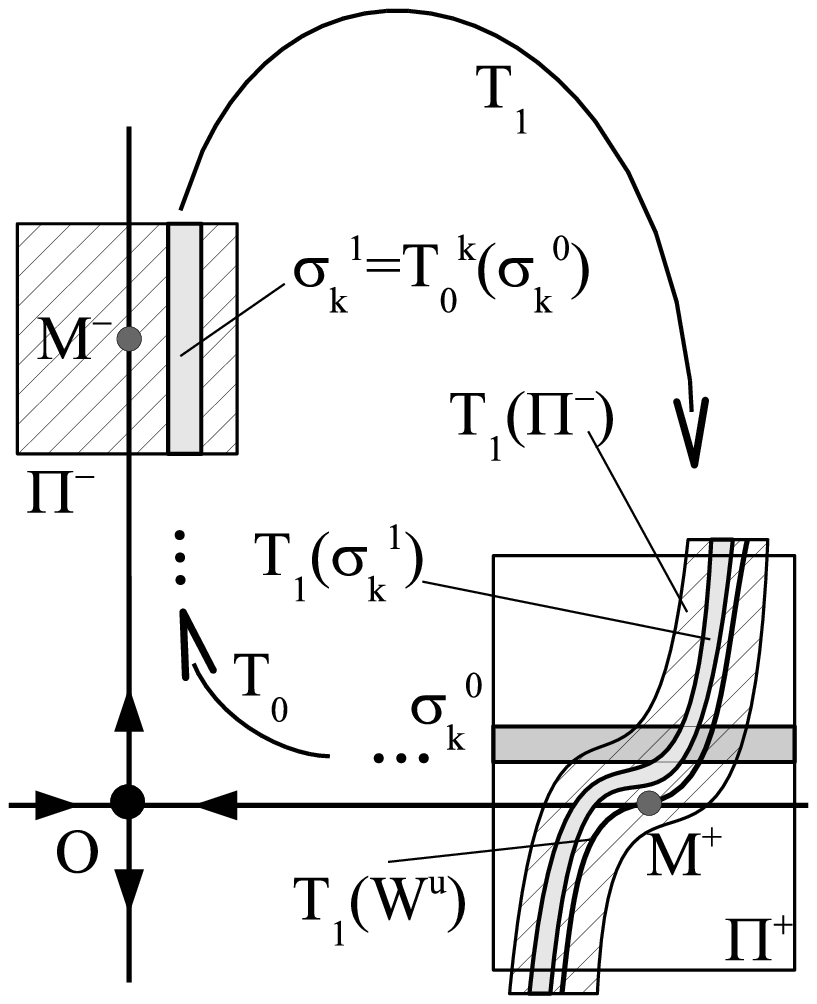}
\caption[]{{ 
(a) An example of planar map
having a cubic tangency at the points of a homoclinic
orbit~$\Gamma_0$. Some of these homoclinic points are shown as
grey circles. Also a small neighborhood of the set $O\cup\Gamma_0$
is shown as the union of the squares;
(b) construction of first return map~$T_k = T_1T_0^k$. }  }
\label{fg:fabhab}
\end{figure}

In the present paper we study bifurcations of
single-round periodic orbits in the families~$f_\varepsilon$ (we keep in mind that $\varepsilon = (\mu_1,\mu_2)$). Note that every
point of such an orbit can be considered as a fixed point of the
corresponding \emph{first return map}. Such a map is usually
constructed as a superposition~$T_k= T_1T_0^k$ of two maps~$T_0\equiv T_0(\varepsilon)$ and $T_1\equiv T_1(\varepsilon)$, see
Figure~\ref{fg:fabhab}(b). Map~$T_0$ is called \emph{local map}
and it is the restriction of~$f_\varepsilon$ onto~$U_0$, i.e. $T_0(\varepsilon)\equiv f_\varepsilon\bigl|_{U_0}$.
Map~$T_1$ is called \emph{global map}, it is defined as
$T_1 \equiv f_\varepsilon^q$ with an integer~$q$ and maps a small neighborhood
$\Pi^-\subset U_0$ of some
point $M^-\in W^u_{loc}(O)$ of orbit~$\Gamma_0$ into a
neighborhood $\Pi^+\subset U_0$ of another
point $M^+\in W^s_{loc}(O)$ of~$\Gamma_0$ 
such that $f_0^q(M^-)=M^+$. Thus, any fixed point of~$T_k$ is a
point of a single-round periodic orbit of period~$k+q$ for~$f_\varepsilon$.

In order to study maps~$T_k$ for all sufficiently
large integer~$k$, it is very important to have good coordinate representations
for both maps~$T_0$ and~$T_1$, especially it concerns local map~$T_0$
and its iterations~$T_0^k$ for large~$k$.

We use in~$U_0$ the canonical coordinates~$(x,y)$ given by the following lemma.

\begin{lmm} {\rm \cite{GG09}}\label{lem:nfgen}
For any given integer~$n$ (such that $n < r/2$ or~$n$ is arbitrary
for $r=\infty$ or $r=\omega$ (the real analytic case)), there is a
canonical change of coordinates, of class~$C^r$ for~$n=1$ or
$C^{r-2n}$ for~$n\geq 2$, that brings~$T_0$ to the following
form\footnote{Note that the smoothness of these changes with
respect to parameters can be $C^{r-2}$ for~$n=1$ or $C^{r-2n-2}$ for~$n\geq 2$,
respectively, see more details in \cite{GST07,book}.}
\begin{equation}\label{eq:fng1}
\begin{array}{l}
\bar x = \lambda x {M_n(xy)}
+ x|xy|^nO\left(|x|+|y|\right),\;\;\\
\bar y = \lambda^{-1} y M_n^{-1}(xy)
+ y|xy|^nO\left(|x|+|y|\right),\;\;
\end{array}
\end{equation}
where $M_n(xy) = 1+\beta_1\cdot xy +\dots + \beta_n\cdot (xy)^n$ and $\beta_i$ are the Birkhoff coefficients.
\end{lmm}

The normal forms (\ref{eq:fng1}) are very suitable for effective
calculations of maps~$T_0^k:(x_0,y_0)\rightarrow (x_k,y_k)$ with
sufficiently large integer $k$. Thus, the following result is valid.

\begin{lmm} {\rm \cite{GG09}} \label{lem:Tknfgen}
Let~$T_0$ be given by (\ref{eq:fng1}), then map~$T_0^k$ can be
written, for any integer~$k$, as follows
\begin{equation}
\begin{array}{l}
x_k = \lambda^k x_0\cdot R_n^{(k)}(x_0y_k,\varepsilon)  +
\lambda^{nk}P_n^{(k)}(x_0,y_k,\varepsilon), \\
y_0 = \lambda^{k} y_k\cdot R_n^{(k)}(x_0y_k,\varepsilon)  +
\lambda^{nk}Q_n^{(k)}(x_0,y_k,\varepsilon),
\end{array}
\label{eq:Tkgen}
\end{equation}
where
$
R_n^{(k)}\equiv 1 + \tilde\beta_1(k)\lambda^k x_0y_k +\dots +
\tilde\beta_n(k)\lambda^{nk} (x_0y_k)^n,
$
$\tilde\beta_i(k)$, $i=1,\dots,n,$ are some {$i$-th degree} polynomials of~$k$ with coefficients
depending on
$\beta_1,\dots,\beta_i,$ and the functions $P_n^{(k)},Q_n^{(k)} =
o\left(x_0^ny_k^n\right)$ are uniformly bounded in~$k$ along with
all their derivatives by coordinates up to order either~$(r-2)$ for~$n=1$ or~$(r-2n-1)$ for~$n\geq 2$.
\end{lmm}

\begin{rmrk}
The main case~$n=1$ of Lemmas~\ref{lem:nfgen}  and~\ref{lem:Tknfgen}, related to the so-called main normal form of a 
saddle map, was proved in the papers~\cite{GS00, GST07}.
 Analogues of Lemmas~\ref{lem:nfgen}  and~\ref{lem:Tknfgen} in the case of
area-preserving and orientation-reversing maps were
proved in~\cite{DGG15}.
\label{rem1}
\end{rmrk}

\begin{rmrk}
Normal form~(\ref{eq:fng1})
can be considered as a certain finitely smooth
approximation of the following analytical Moser normal form 
\begin{equation}
\begin{array}{l}
\bar x = \lambda(\varepsilon)x\cdot M(xy,\varepsilon),\;\; \bar y
= \lambda^{-1}(\varepsilon) y\cdot M^{-1}(xy,\varepsilon),
\end{array}
\label{eq:BMnf}
\end{equation}
which takes place for $\lambda>0$~\cite{Mos56},  where
$M(xy,\varepsilon)= 1+ \sum\limits_{i=1}^{\infty} \beta_i\cdot (xy)^i$.
Accordingly, relation~(\ref{eq:Tkgen})
looks as a very good approximation for the corresponding formulas in
the analytical case,~\cite{GS97},
$$
x_k = \lambda^k x_0\cdot R^{(k)}(x_0y_k,\varepsilon),\;\; y_0 =
\lambda^{k} y_k\cdot R^{(k)}(x_0y_k,\varepsilon),
$$
where $R^{(k)}\equiv 1 + \sum\limits_{i=1}^{\infty}\tilde\beta_i(k)\lambda^{nk} (x_0y_k)^i $ and
$
\tilde\beta_1(k) = \beta_1 k,\;\; \tilde\beta_2(k) = \beta_2 k -
\frac{1}{2}\beta_1^2 k^2,
$ etc.
\label{rem2}
\end{rmrk}

In the coordinates of Lemma~\ref{lem:nfgen}, we have that
$M^+=(x^+,0),M^-=(0,y^-)$.  Without loss of generality, we assume
that~$x^{+}>0$ and~$y^{-}>0$. Let neighborhoods~$\Pi^{+}$ and
$\Pi^{-}$ of the homoclinic points~$M^{+}$ and~$M^{-}$,
respectively, be sufficiently small such that
$T_0(\Pi^+)\cap\Pi^+=\emptyset$ and
$T_0^{-1}(\Pi^-)\cap\Pi^-=\emptyset$. Then, as usually (see e.g.~\cite{GS73}), the successor map from~$\Pi^+$ 
into~$\Pi^-$ by orbits of~$T_0$ is defined, for all sufficiently small
$\varepsilon$,  on the set consisting of infinitely many strips
$\sigma_k^0\equiv \Pi^+\cap T_0^{-k}\Pi^-$, $k=\bar k,\bar
k+1,\dots$. The image of~$\sigma_k^0$ under~$T_0^k$ is the strip
$\sigma_k^1= T_0^k(\sigma_k^0)\equiv \Pi^-\cap T_0^{k}\Pi^+$. As~$k\to\infty$, the strips~$\sigma_k^0$ and~$\sigma_k^1$ accumulate
to $W^s_{loc}$ and $W^u_{loc}$, respectively.

We can write  global map~$T_{1}(\varepsilon):\Pi^{-}\rightarrow \Pi^{+}$ as follows (in the
coordinates of Lemma~\ref{lem:nfgen})
\begin {equation}
\begin {array}{l}
\overline{x} -x^{+}  =  F(x,y-y^{-},\varepsilon),\;\;
\overline{y} = G(x,y-y^{-},\varepsilon),
\end {array}
\label{eq:t1}
\end{equation}
where $F(0,0,0)=0,G(0,0,0)=0$ and 
\begin {equation}
\frac{\partial G(0)}{\partial y}=0,\;\; \frac{\partial^2
G(0)}{\partial y^2}=0,\;\;\frac{\partial^3 G(0)}{\partial y^3}=
6d\neq 0.
\label{eq:t1-cubt}
\end{equation}
These relations follow from condition {\bf (B)}, which means that for $\varepsilon=0$
curve~$T_1(W^u_{loc}):\{\overline{x} -x^{+} = F(0,y-y^{-},0),
\overline{y} = G(0,y-y^{-},0)\}$ has a cubic tangency with
$W_{loc}^s:\{\bar y =0\}$ at $M^+$. When the parameters vary, this
tangency can split and, by condition {\bf (C)}, family~(\ref{eq:t1}) unfolds generally the initial cubic tangency. In
this case,
global map~$T_1$ can be written in a certain normal form that
the following lemma shows.
\begin{lmm}
If condition {\bf (C)} holds, then map~$T_1(\varepsilon)$ can
be brought to the following form
\begin {equation}
\begin {array}{l}
\overline{x} -x^{+}  =  a x + b (y-y^{-}- ) + O\left(x^2 +
(y-y^-)^2\right),\\
\overline{y} = \mu_1 + \mu_2(y-y^{-}) + cx +
d (y-y^{-})^3 + O\left(x^2 + |x||y-y^-| + (y-y^-)^4
\right),
\end {array}
\label{eq:t1par}
\end{equation}
where
$$
b(\varepsilon)c(\varepsilon)\;\equiv\;-1
$$
and the coefficients $a,b,c,d$
as well as~$x^+,y^-$ depend smoothly (with the same smoothness
as for the initial map~(\ref{eq:t1})) on new parameters
$\varepsilon$ 
such that $\mu_1\equiv\varepsilon_1$ and $\mu_2\equiv\varepsilon_2$ (the other
parameters are not essential).
\label{nfT1}
\end{lmm}

\begin{proof} By virtue of (\ref{eq:t1-cubt}), {when~$x$ and~$\varepsilon$ are small enough},
equation~$\partial^2G(x,y-y^-,\varepsilon)/\partial y^2 =0$ can be solved for $y-y^-$. The solution
has the form $y-y^-(\varepsilon)=\varphi(x,\varepsilon)$, where
$\varphi(0,\varepsilon)\equiv 0$. Then, we can write the following
Taylor expansion
for function~$G$ near
curve~$y-y^-(\varepsilon)=\varphi(x,\varepsilon)$:
$$
 G\;\equiv\;G(x,0,\varepsilon) + \frac{\partial
G(x,0,\varepsilon)}{\partial y} (y-y^{-}-\varphi) +
\frac{\partial^3 G(x,0,\varepsilon)}{\partial y^3}
(y-y^{-}-\varphi)^3 + O\left((y-y^- -\varphi)^4\right).
$$
Since $\varphi\equiv \varphi(x,\varepsilon) = O(x)$, we obtain
that
$$
\begin{array}{l}
G\;\equiv\;E_1(\varepsilon) + cx +  E_2(\varepsilon)(y-y^{-}) +
d (y-y^{-})^3 + 
O\left(x^2 +
|x||y-y^-|  + (y-y^-)^4 \right),
\end{array}
$$
where~$E_i(0)=0$. Hence,  map~$T_1(\varepsilon)$ can be written
in the following form
$$
\begin {array}{l}
\overline{x} -x^{+}  =  a x + b (y-y^{-}) + O\left(x^2 +
(y-y^-)^2\right),\\
\overline{y} = E_1(\varepsilon) + c x + E_2(\varepsilon)(y-y^{-})
+
d (y-y^{-})^3 +
O\left(x^2 +
|x||y-y^-| + (y-y^-)^4 \right).
\end {array}
$$
Then the equation of the piece of $T_1(W^u_{loc})$ in~$\Pi^+$ has the
following form
\begin {equation}
\begin {array}{l}
\overline{x} -x^{+}  =   b (y-y^{-}) + O\left((y-y^-)^2\right),\\
\overline{y} = E_1(\varepsilon) + E_2(\varepsilon)(y-y^{-}) +
d(y-y^{-})^3 +  O\left((y-y^-)^4\right).
\end {array}
\label{eq:t1Wu}
\end{equation}
Condition {\bf(C)} implies that $E_i(0)=0, E_i^\prime(0)\neq 0$ and
the coefficients~$E_1(\varepsilon)$ and~$E_2(\varepsilon)$ can
take any values from the ball $\|\varepsilon\|\leq\delta_0$, where
$\delta_0>0$ is a small constant. Thus, the system
$\mu_1=E_1(\varepsilon),\mu_2=E_2(\varepsilon)$ has always a
solution and we can consider~$\mu_1$ and~$\mu_2$ as new
parameters. 
\end{proof}

Due to Lemma~\ref{nfT1}, we can consider directly  two parameter
families~$f_{\mu_1,\mu_2}$ of symplectic maps.  It is easy to see
from~(\ref{eq:t1par}) that~$\mu_1$ and~$\mu_2$ are
the splitting parameters of the manifolds $W^s(O_\varepsilon)$ and
$W^u(O_\varepsilon)$ with respect to the homoclinic point~$M^+$.
Indeed, by (\ref{eq:t1Wu}), curve~$l_u=
T_1(W^u_{loc}\cap\Pi^-)$ has the equation
\begin {equation}
\displaystyle l_{u}\;:\; y= \mu_1 + \frac{\mu_2}{b}(x-x^{+}) +
\frac{d}{b^{3}}(x-x^{+})^3 + o\left((x-x^{+})^3\right).
\label{eq:T1Wuloc}
\end{equation}
Thus, family~$f_{\mu_1,\mu_2}$ is a general two parameter
unfolding of the initial cubic tangency which occurs for
$\mu_1=\mu_2=0$.
For any such unfolding, curves~$W^s_{loc}$ and
$T_1(W^u_{loc})$ must have a quadratic tangency for certain values
of $\mu_1$ and $\mu_2$. It is true that for our family the
following result, which implies item 1 of Theorem~\ref{thm:main}, holds.
\begin{figure}[tb]
  \centerline{
  \includegraphics[width=12cm]{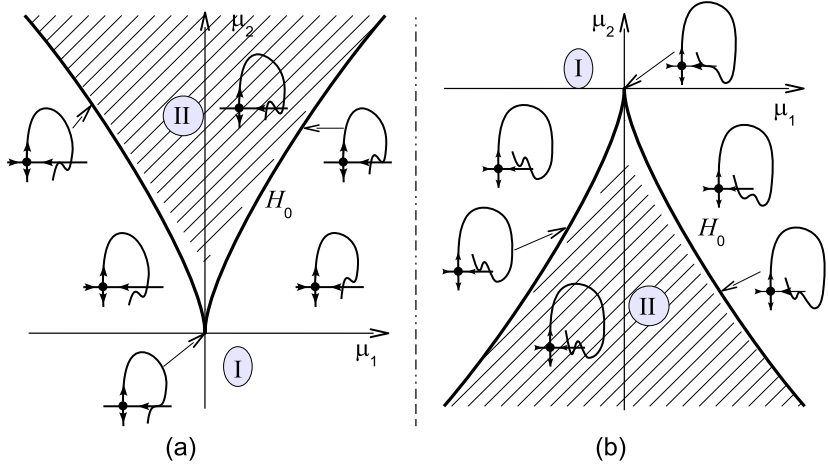}}
  \caption{{ Curve~$H_0$ in (a) incoming from above ($d<0$) and
  (b) incoming from below ($d>0$)    cases.  For $\mu\in I$, map~$f_\mu$ has only one
  transverse homoclinic orbit close to $\Gamma_0$, whereas, for $\mu\in II$, $f_\mu$ has exactly 3 transverse 
  homoclinic orbits.}}
\label{fig18-01}
\end{figure}

\begin{prpstn}
In the $(\mu_1,\mu_2)$-parameter plane, there exists a
bifurcation curve~$H_0$:
$$
\mu_1 = \pm 2d\left[-\frac{\mu_2}{3d}(1+O\left(\sqrt{|\mu_2|}\right))\right]^{3/2}
$$
(see Figure~\ref{fig18-01})
such that for $\mu\in H_0$, map~$f_\mu$ has a close to $\Gamma_0$  homoclinic orbit
with a quadratic tangency of the manifolds $W^u(O _\mu)$ and $W^s(O_\mu)$.
\label{th1g85}
\end{prpstn}

\begin{proof} Consider curve~$l_u$ given by equation~(\ref{eq:T1Wuloc}). If the curve has a tangency with the line
$y=0$, the following system has solutions
$$
\begin{array}{l}
\displaystyle \mu_1 + \frac{\mu_2}{b}(x-x^{+}) +
\frac{d}{b^{3}}(x-x^{+})^3 +
o\left((\bar x-x^{+})^3\right) =0, \\
\displaystyle \mu_2 + \frac{3d}{b^2} (x-x^+)^2 + o\left((\bar
x-x^{+})^2\right) =0.
\end{array}
$$
We solve the second equation for $(x-x^+)$, put it into the first equation, and we obtain
the equation of curve~$H_0$. Note that the tangency is always
 quadratic except for the case $x-x^+=0$ which corresponds to a cubic tangency at~$\mu_1=\mu_2=0$.
\end{proof}

\section{On bifurcations of single-round periodic orbits}\label{sect:bifcub}

The main {goal in the present paper} is
to study bifurcations of single-round periodic orbits
of family~$f_{\mu_1,\mu_2}$. As said above, it is equivalent to the study of bifurcations of
fixed points in the first return maps~$T_k=T_1T_0^k$ for every sufficiently large integer~$k$ ($k=k_0,k_0+1,\ldots$). 
We apply the
rescaling method {in order} to find the rescaled normal forms for these maps in the following lemma.

\begin{lmm}
{\sf [Rescaling Lemma for cubic homoclinic tangencies]}\\
For every sufficiently large $k$, the first return map~$T_k$ can be
brought, by a linear transformation of coordinates and parameters,
to the following form
\begin{equation}
\begin{array}{l}
\bar{x} \; = \; y + O({|\lambda|^{k/2}}), \\
\displaystyle \bar{y} \;=\; M_1 - x + M_2y + \nu y^3 +  O({|\lambda|^{k/2}}),
\end{array}
\label{henon0-3}
\end{equation}
where
$$
\nu\; =\; \mbox{{\sf sign}}\;(d\lambda^k),
$$
\begin{equation}
\begin{array}{l}
\displaystyle M_1 = \sqrt{|d|}{|\lambda|^{-k/2} \lambda^{-k}} \left(
\mu_1 - \lambda^k(y^- -cx^+) + O\left(k\lambda^{2k}\right)\right),\\
M_2 = \lambda^{-k}\left(\mu_2 + f_{11}\lambda^k x^+ + O\left(k\lambda^{2k}\right)\right)
\end{array}
\label{mui-c}
\end{equation}
and $f_{11} = G_{xy}(0)$.
 \label{henmain-c}
\end{lmm}

\begin{proof} By Lemmas~\ref{lem:Tknfgen} and~\ref{nfT1}, {the} first return map~$T_k=T_1T_0^k$ can be written as follows
\begin{equation}
\begin{array}{rcl}
\bar x - x^+ &=& \; a{\lambda}^kx + b(y-y^-) + O\left(k\lambda^{2k} |x| + |\lambda|^k |x||y-y^-| +
(y-y^-)^2\right),  \\
\lambda^{k}\bar y (1 + k \lambda^k O(\bar x,\bar y)) &=& \;\mu_1 + \mu_2(y_1-y^-) + d(y_1-y^-)^3 +
c\lambda^kx  \\
&&+\; O\left(k\lambda^{2k} |x| + |\lambda|^k |x||y-y^-| \right) +
O\left((y-y^-)^4\right)
\end{array}
\label{tk}
\end{equation}
We shift the coordinates,
$\xi = x-x^+ + \alpha_k^1, \eta = y-y^- + \alpha_k^2$, where $\alpha_k^1 = - a\lambda^kx^+ +
O(k\lambda^{2k})$,  $\displaystyle \alpha_k^2 = - \frac{f_{12}}{3d}\lambda^kx^+ + O(k\lambda^{2k})$
and $\displaystyle f_{12}=\frac{1}{2}G_{xy^2}(0)$. Then  system (\ref{tk}) is rewritten as
follows
\begin{equation}
\begin{array}{rcl}
\bar \xi &=& a{\lambda}^k\xi  + b\eta + O\left(k\lambda^{2k} |\xi| + |\lambda|^k |\xi||\eta| + \eta^2\right), \\
 {(1 + k\lambda^k O(\bar\eta))}
\bar\eta & =& \lambda^{-k}\left(\mu_1 - \lambda^k(y^- - cx^+) +  {\rho_k^{1}
}\right)
+ \lambda^{-k}\left(\mu_2 + f_{11}\lambda^k x^+ +  {\rho_k^{2}}
 \right)\eta \\
&&+ d\lambda^{-k}\eta^3 + cx +
O\left(k\lambda^{k} |\xi| + |\xi||\eta| \right) + \lambda^{-k}O\left(\eta^4\right),
\end{array}
\label{tk+}
\end{equation}
where $\rho_k^{1,2}=O\left(k\lambda^{2k}\right)$ are small {coefficients}.
Here  the first  equation of (\ref{tk+}) does not contain constant terms and the quadratic term in
$\eta^2$  vanishes in the second equation.

Finally, by means of the coordinate rescaling
$$
\xi \;=\; b\frac{|\lambda|^{k/2}}{\sqrt{|d|}}\;\; x, \;\; \eta \;=\; \frac{|\lambda|^{k/2}}{\sqrt{|d|}}\;\; y
$$
we bring map~(\ref{tk+}) to the required form~(\ref{henon0-3}).
\end{proof}

\subsection{The description of bifurcations of fixed points in the cubic H\'enon maps}\label{sect:cubicHenon}

By the Rescaling Lemma~\ref{henmain-c}, the following conservative cubic
H\'enon maps~$\mathbf{C}_+^1 : \bar{x} = y, \; \bar{y} = M_1 - x + M_2y + y^3$ 
and~$\mathbf{C}_-^1 : \bar{x} = y, \; \bar{y} = M_1 - x + M_2y - y^3$
have to be considered as certain normal forms for the first return
maps in the cases $d\lambda^k>0$ and $d\lambda^k<0$, respectively. Thus, if $\lambda>0$, map $\mathbf{C}_+^1$
corresponds to the cubic homoclinic tangency with~$d>0$ (the { incoming from below} tangency, see Figure~\ref{fig2type}(b)),
while map~$\mathbf{C}_-^1$  is related
to the cubic homoclinic tangency with $d<0$ (the {incoming from above} tangency, see Figure~\ref{fig2type}(a)).
In Figures~\ref{GK88-+2} and~\ref{GK88--2}, the main elements of
the bifurcation diagrams for these cubic maps are presented.

The bifurcation curves in Figures~\ref{GK88-+2} and~\ref{GK88--2} are found analytically (see e.g.~\cite{GK88,DM00}) 
and their
equations are as follows, where~$\nu=+1$ and~$\nu=-1$ correspond to the maps~$\mathbf{C}_+^1$ and~$\mathbf{C}_-^1$, 
respectively:
\begin{equation}
\begin{array}{l}
\displaystyle L^+:\;\; M_1^2  = \frac{4}{27\nu}
 \left(2-M_2\right)^3  ; \\ \\
\displaystyle L^-:\;\; M_1^2 = - \frac{4}{27\nu}
\left(2+M_2\right)\left(4-M_2\right)^2 ;\\ \\
\displaystyle L_{2}^+:\;\; M_1^2 =  - \frac{4}{27}
\left(M_2+4\right)^3, \;\; \mbox{ if }\;\; \nu=+1 ;\\ \\
\displaystyle L_{2}^{+i}, i=1,2:\;\; M_1 = (-1)^i\;
2\left(\frac{M_2+4}{3}\right)^{3/2},\;\; M_2 > -\frac{4}{3}, \;\;
\mbox{ if }\;\; \nu=-1 ;
\\ \\
\displaystyle L_{2}^{-i}, i=1,2:\;\; M_1^2 = \frac{1}{216\nu}\left[ 12 +
M_2 + (-1)^i \sqrt{9M_2^2 + 24M_2}\right]^2\left[-5M_2 -12 + (-1)^i \sqrt{9M_2^2 + 24M_2}\right];
\end{array}
\label{eq:criv}
\end{equation}
where 
if $(M_1,M_2)\in L^+$, then there exists
a fixed point with double {eigenvalue} $+1$; if $(M_1,M_2)\in L^-$, there exists a fixed point with double {eigenvalue}~$-1$;
if $(M_1,M_2)\in L_{2}^+$, there exists a 2-periodic orbit with double {eigenvalue}~$+1$;
if~$(M_1,M_2)\in L_{2}^-$,
there exists a 2-periodic orbit with double {eigenvalue}~$-1$.

The main bifurcations related to fixed points are as follows.

$\;$\\
\textbf{Bifurcation scenario in map~$\mathbf{C}_+^{\;1}$, see Figure~\ref{GK88-+2}.} 
Bifurcation curves~$L^+$ and~$L^-$ from~(\ref{eq:criv}) with~$\nu =+1$ divide the~$(M_1,M_2)$-parameter plane into 3 regions
symmetric with respect to the~$M_2$-axis.
{When parameters $M_1$ and $M_2$ are in region~${\bf 1}$},
map~$\mathbf{C}_+^{\;1}$  has a unique fixed point~$p_1$ which is a saddle-plus
(with {eigenvalues} $\lambda$ and $\lambda^{-1}$, where $0<\lambda<1$).
The transition ${\bf 1}\Rightarrow{\bf 2}$ corresponds to the  {birth} of two (saddle and elliptic)
fixed points.
At $(M_1,M_2) = b_1$, the fixed point $p_1$ is a non-hyperbolic saddle with double {eigenvalue}~$+1$,
and then this point falls into three (two saddle and one elliptic) fixed points  in domain~${\bf 2}$
under a conservative pitch-fork bifurcation.
The transition ${\bf 2}\Rightarrow{\bf 3}$ corresponds to a nondegenerate period doubling
bifurcation of the elliptic fixed point. Thus, for region ${\bf 3}$, map~$\mathbf{C}_+^{\;1}$ has
three saddle (two saddle-plus and one saddle-minus) fixed points  and one elliptic 2-periodic orbit.
Further primary bifurcations, when crossing curves~$L_2^+$ and~$L_2^-$, are related  
to conservative fold and period doubling bifurcations of  2-periodic orbits, respectively,
and we do not  discuss them here (see more details in~\cite{GK88,DM00}).
\begin{figure} 
  \centerline{
  \includegraphics[width=16cm]{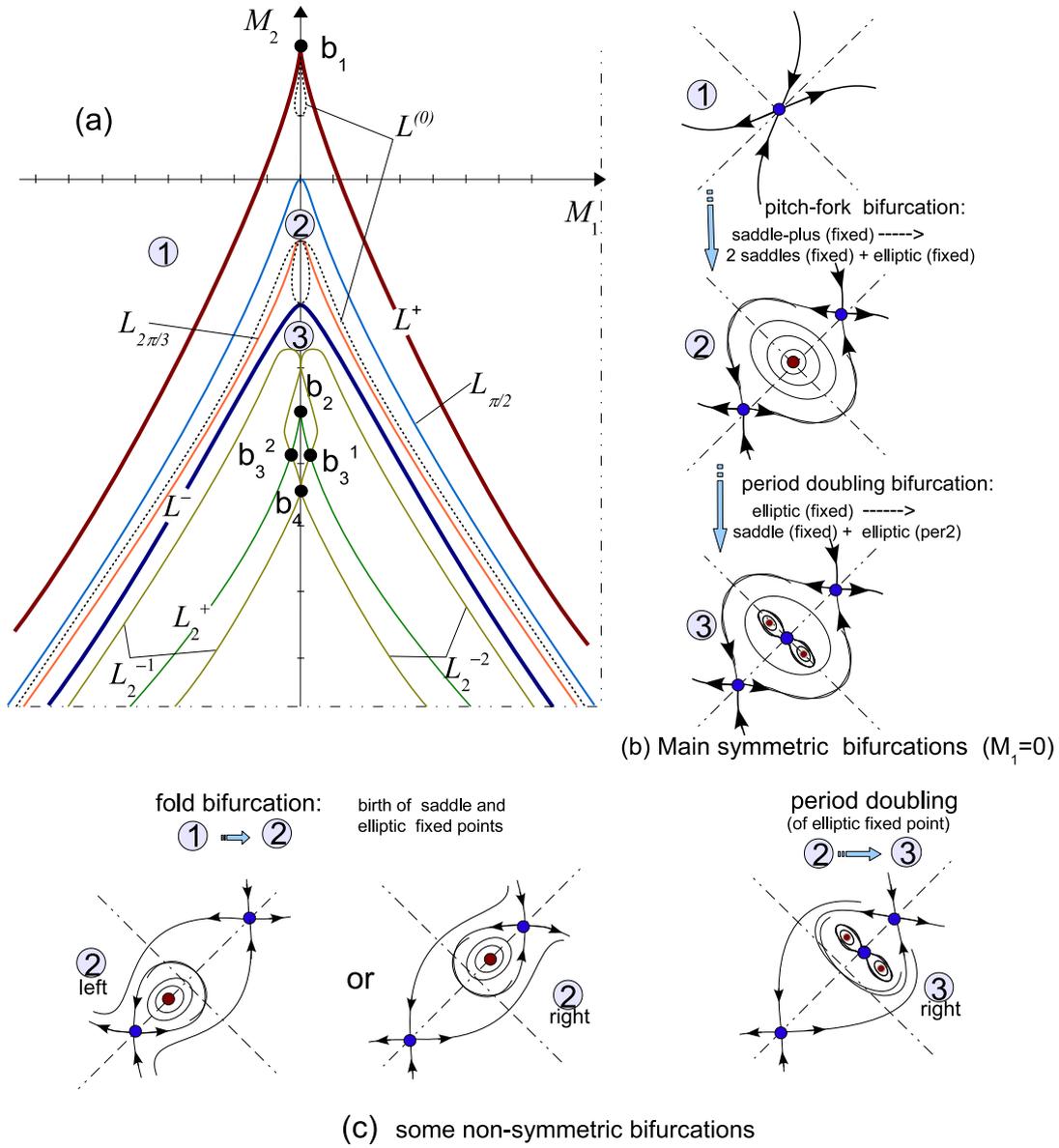}}
  \caption{{{
(a) The main elements of the bifurcation diagram for map~$\mathbf{C}_+^{\;1}$. The equations of bifurcation curves~$L^+$,
$L^-$ are given by~(\ref{eq:criv}) with~$\nu=1$. The equation of~$L^{(0)}$ is given 
by~(\ref{eq:twistless}) with~``+''. 
The codimension 2 bifurcation
points~$b_i$
are as follows: $b_1$  --  a  nonhyperbolic saddle fixed point with double {eigenvalue}~$+1$ exists;
  $b_2$ and $b_4$  --  two 2-periodic cycles  with double {eigenvalue}~$-1$ coexist;
  $b_3^{1,2}$  --  two 2-periodic cycles, one with double {eigenvalues}~$-1$ and the other with double {eigenvalues}~$+1$, coexist.
  Examples of (b) symmetric ($M_1=0$) and
(c) asymmetric~$(M_1\neq 0)$ bifurcations are also shown.}} }
   \label{GK88-+2}
\end{figure}

\begin{figure}
  \centerline{
  \includegraphics[width=16cm]{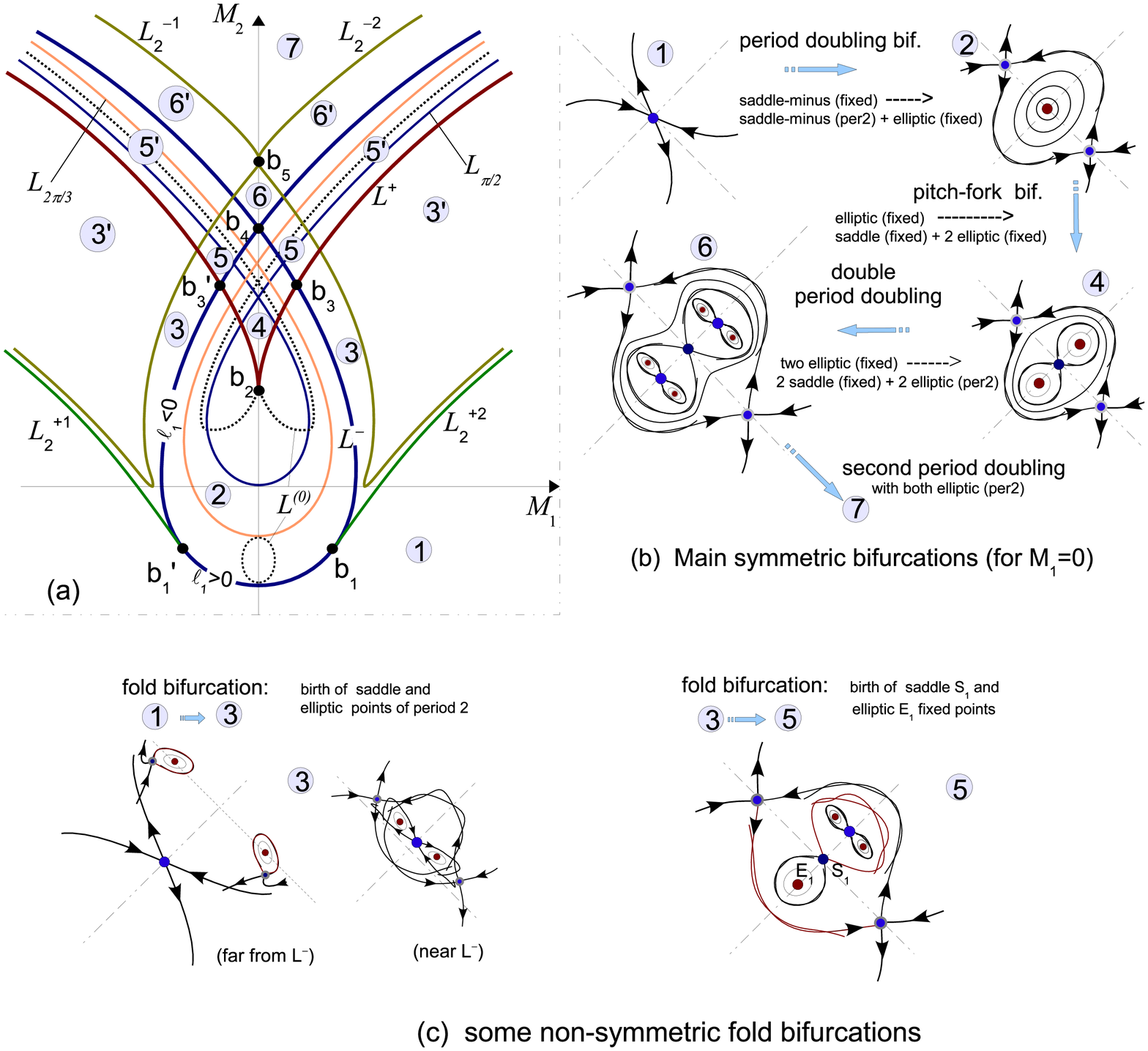}}
  \caption{{ {(a) The main elements of the bifurcation diagram for map~$\mathbf{C}_-^{\;1}$. 
  The equations of bifurcation curves~$L^+$ and~~$L^-$ are given by~(\ref{eq:criv}) with~$\nu=-1$. The equation of $L^{(0)}$ is given 
by~(\ref{eq:twistless}) with~``$-$''.
The codimension 2 bifurcation points~$b_i$
  are as follows: $b_1$ and $b_1^\prime$   --
  there exists a fixed point with {eigenvalues}~$(-1;-1)$ whose first Lyapunov value is zero but 
  second is nonzero, these points divide curve~$L^-$ into the segments where the corresponding period doubling bifurcation is either subcritical (when crossing the piece $(b_1,b_1^\prime)$ of~$L^-$) or supercritical (crossing outside $(b_1,b_1^\prime)$); $b_2$  -- a triple fixed
  point of elliptic type exists; $b_3$  and $b_3^\prime$  -- two fixed points with {eigenvalues}
  $(-1,-1)$  and $(+1,+1)$ coexist; $b_4$  --  two fixed points with {eigenvalues}
  $(-1,-1)$ coexist; $b_5$ -- two 2-periodic orbits with {eigenvalues}
  $(-1,-1)$ coexist.}} }
   \label{GK88--2}
\end{figure}

$\;$\\
\textbf{Bifurcation scenario in map~$\mathbf{C}_-^{\;1}$, see Figure~\ref{GK88--2}.} 
Bifurcation curves~(\ref{eq:criv}) with~$\nu = -1$
divide the~$(M_1,M_2)$-parameter plane into 15 regions including the regions ${\bf 1}$, $\textbf{2}$, $\textbf{4}$, 
$\textbf{6}$ and
$\textbf{7}$ which are symmetric with respect
to the~$M_2$-axis, while the other regions are pairwise symmetric (e.g $\textbf{3}$ and $\textbf{\underline3}$).
For {$(M_1,M_2)$ belonging to region~${\bf 1}$}, map~$\mathbf{C}_-^{\;1}$ has a unique fixed point~$q_1$
 which is a saddle-minus
(with {eigenvalues} $\lambda$ and $\lambda^{-1}$, where $-1<\lambda<0$). The transition
${\bf 1}\Rightarrow{\bf 2}$ through the segment $(b_1,b_1^\prime)$ of curve~$L^-$ corresponds
to a subcritical period doubling bifurcation with the point~$q_1$: the point becomes elliptic
and a saddle 2-periodic orbit emerges in its neighborhood. The transition ${\bf 1}\Rightarrow{\bf 3}$ (and symmetrically ${\bf 1}\Rightarrow{\bf \underline{3}}$) corresponds to a nondegenerate conservative fold 
bifurcation under which saddle and elliptic 2-periodic orbits are born.
The transition ${\bf 3}\Rightarrow{\bf 2}$ (as well as ${\bf \underline{3}}\Rightarrow{\bf 2}$) corresponds to
a supercritical period doubling bifurcation under which the elliptic 2-periodic orbit merges with the saddle fixed point and
becomes an elliptic fixed point.
Thus, in domain~${\bf 2}$, map~$\mathbf{C}_-^{\;1}$
has an elliptic fixed point and  a saddle 2-periodic orbit.
The transitions ${\bf 3}\to {\bf 3'}, {\bf 5}\to {\bf 5'}$ and ${\bf 6}\to {\bf 6'}$  (and, symmetrically, 
${\bf\underline 3}\to {\bf\underline{ 3'}}, {\bf \underline5}\to {\bf\underline{ 5'}}$ and 
${\bf 6}\to {\bf\underline{ 6'}}$) correspond to a period doubling of the elliptic 2-periodic orbit 
(it becomes a saddle 2-periodic orbit and an elliptic 4-periodic orbit is born in its neighborhood).
We also illustrate bifurcations
occurring at a passage {around} the point $b_1$ in Figure~\ref{GM--2} (and, symmetrically, for
the point~$b_1^\prime$).
 Transitions through curve~$L^+$ at increasing~$M_2$
correspond to the appearance of
two new (saddle and elliptic) fixed points under a conservative fold bifurcation. Thus, in region~${\bf 4}$, 
map~$\mathbf{C}_-^{\;1}$ has three fixed points, two elliptic and one saddle.
The elliptic fixed points undergo a period doubling bifurcation at transitions
${\bf 4}\Rightarrow{\bf\underline 5}$ and ${\bf 5'}\Rightarrow{\bf 6'}$ -- for one of the points, or  
${\bf 4}\Rightarrow{\bf 5}$ and ${\bf\underline{5'}}\Rightarrow{\bf\underline{ 6'}}$ -- for the other point.
We note that for  $(M_1,M_2)=b_2$, there exists a triple (of elliptic type) fixed
 point which splits into three (two elliptic and one saddle) fixed points in region~${\bf 4}$.
\begin{figure} [tb]
  \centerline{
  \includegraphics[width=12cm]{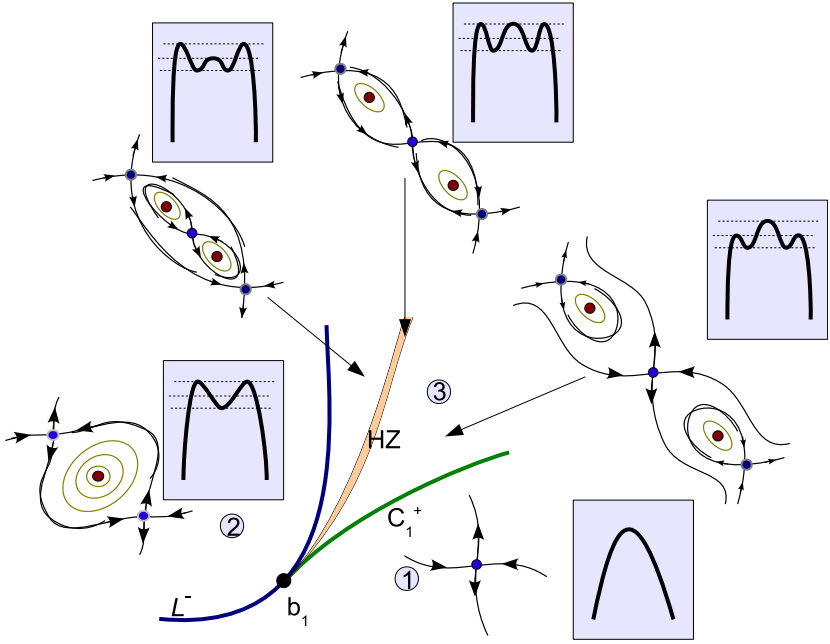}}
  \caption{{ Bifurcations near point~$b_1$ of the bifurcation diagram in Figure~\ref{GK88--2}.
Region~$HZ$ (homo/heteroclinic zone) corresponds to the values of~$(M_1,M_2)$ where the invariant manifolds of
all saddles intersect. In a rough approximation, these bifurcations are similar to those of
a two-dimensional Hamiltonian system whose potential function is symmetric and changes  as shown.}}
   \label{GM--2}
\end{figure}

\subsection{{On genericity of elliptic fixed points for the conservative cubic H\'enon maps.}} \label{sec:genellpoint}

In the~$(M_1,M_2)$-parameter plane, bifurcation curves~$L^+$ and~$L^-$ bound those open regions of parameter values 
where elliptic fixed
points of  maps~$\mathbf{C}_+^{\;1}$ and~$\mathbf{C}_-^{\;1}$ exist. Such a fixed point is called {\em generic} if, first, 
$\varphi\neq 0,\pi,\pi/2,2\pi/3$, i.e.
no strong resonances occurs, and, second, in the main complex normal form
\begin{equation}
\bar z = e^{i\varphi} z + i B_1 z |z|^2 + O(|z|^4)
\label{eq:ellgen}
\end{equation}
the first Birkhoff coefficient~$B_1$ is nonzero. Note that if an elliptic point is generic, then it is stable. However,
non-generic elliptic points can be both stable (of elliptic type) and unstable (of saddle type, for example, a saddle 
with 6 or 8 separatrices
if $\varphi = 2\pi/3$ or $\varphi = \pi/2$). 

The following curve
\begin{equation}
L_\varphi\;:\; M_1^2 = \pm \frac{4}{27}(M_2+\cos\varphi -3)^2(2\cos\varphi - M_2)
\label{eq:ellfp}
\end{equation}
in the~$(M_1,M_2)$-parameter plane defines those parameter values where map~$\mathbf{C}_\pm^{\;1}$ has a fixed point 
$E_\varphi$ with {eigenvalues} $e^{\pm i\varphi}$ with $0<\varphi<\pi$. Accordingly, curves~$L_{2\pi/3}$ and~$L_{\pi/2}$  
of strong 1:3 and 1:4 resonances are resulted from~(\ref{eq:ellfp}) for $\cos\varphi =0$ and $\cos\varphi =-1/2$, respectively.

To find the conditions when the point $E_\varphi$ is non-generic, i.e. $B_1(E_\varphi)=0$,  we use 
the formula for $B_1$ 
from \cite{TurVered}, see formula~(35) there. Then we obtain that $B_1=0$ for
\begin{equation}
M_2 = \frac{6\cos^2\varphi + 3\cos\varphi +1}{1+ 4\cos\varphi}.
\label{eq:ellb0}
\end{equation}
This formula is true for both maps~$\mathbf{C}_{+}^{\;1}$ and~$\mathbf{C}_-^{\;1}$.

Equations~(\ref{eq:ellfp}) and~(\ref{eq:ellb0}) define in the~$(M_1,M_2)$-parameter plane 
the so-called non-twist curve~$L^{(0)}$ such that for~$(M_1,M_2)\in L^{(0)}$, map~$\mathbf{C}_\pm^{\;1}$ 
has a degenerate elliptic fixed point (with~$B_1=0$). These curves are shown in Figures~\ref{GK88-+2} and~\ref{GK88--2}. 
For both maps, the curves consist of two connected pieces, the infinite branches of~$L^{(0)}$ asymptotically tend  
to the curve~$L_\varphi$ with $\varphi = \arccos (-1/4)$. In the case of map~$\mathbf{C}_-^1$, one of the branches of
the curve $L^{(0)}$
has a self-intersection point $(M_1=0; M_2 = 3.2)$,  
where map~$\mathbf{C}_{-}^{\;1}$ 
has simultaneously two degenerate elliptic fixed points (both with~$\cos\varphi = -1/5$).
For map~$\mathbf{C}_+^1$, there also exists a self-intersection point ($M_1=0,M_2=-1$) of curve~$L^{(0)}$. This point corresponds to the existence of a fixed
point with {eigenvalues}~$e^{\pm i 2\pi/3}$ (1:3 resonance), which is KAM-stable in this case,~\cite{DM00}. Note that, in the case of map~$\mathbf{C}_-^1$, the fixed point~$O(0,0)$ at $M_1=0,M_2=-1$ has {eigenvalues}~$e^{\pm i 2\pi/3}$ and is also KAM-stable (here~$L^{(0)}$ and~$L_{2\pi/3}^-$ are tangent). 
Note that the fixed point~${ E}_\varphi$ is always nondegenerate (generic) if $-{1}/{4}\cos\varphi < {1}/{6}$ in 
the case of map~$\mathbf{C}_{+}^{\;1}$ and if $-{1}/{2}<\cos\varphi < -{1}/{4}$  in the case of map~$\mathbf{C}_{-}^{\;1}$.
We also note that curves~$L^{(0)}$ can be written in the explicit form~\cite{DM00} as follows
\begin{equation}
729 M_1^4 \pm (8 M_2^3 - 108 M_2^2 - 63 M_2 + 837) M_1^2 - \frac{16}{27} (M_2 + 1)(5M_2 -16)^2(M_2 - 2)^3 = 0.
\label{eq:twistless}
\end{equation}

\subsection{End of proof of  Theorem~\ref{thm:main}} \label{sec:endofproof}

Now we can complete the proof of Theorem~\ref{thm:main}. Indeed, we only need to translate the results on 
the bifurcation diagrams obtained in the~$(M_1,M_2)$-plane for the rescaling maps~$\mathbf{C}_+^{\;1}$ and~$\mathbf{C}_-^{\;1}$ 
into the initial~$(\mu_1,\mu_2)$-parameter plane. We can easily make this using relations~(\ref{mui-c}) 
between the initial and rescaling parameters. Then, using
equations~(\ref{eq:criv}) for curves~$L^+$, $L^-$ and~$L_{2}^{+}$, we obtain the following equations for the curves in 
the~$(\mu_1,\mu_2)$-plane, where $\nu = \mbox{{\sf sign}}(d\lambda^k)$,
\begin{equation}
\begin{array}{l}
\displaystyle L_k^+:\;\; \mu_1  = \lambda^k(y^- -cx^+ +\dots) \pm
\frac{2}{\sqrt{|d|}} \left(\frac{(2-f_{11}x^+)\lambda^k-\mu_2}{3\nu}\right)^{3/2}(1+\dots);\\ 
\displaystyle L_k^-:\;\; \mu_1  = \lambda^k(y^- -cx^+ +\dots) \pm
\frac{2}{3\sqrt{|d|}}\left(\frac{-(2+f_{11}x^+)\lambda^k-\mu_2}{3\nu}\right)^{1/2} 
 \left((4-f_{11}x^+)\lambda^k-\mu_2\right)(1+\dots);\\
\displaystyle L_{2k}^{+}:\;\; \mu_1  = \lambda^k(y^-
-cx^+\dots)\pm \frac{2}{\sqrt{|d|}}  \left(\frac{-(4+f_{11}x_1^+)\lambda^k-\mu_2}{3}\right)^{3/2}(1+\dots), 
\text{ if }\;\; 
\nu=+1;   \\
\displaystyle L_{2k}^{+i}: \left\{ \begin{array}{l}
\displaystyle \mu_1  = \lambda^k(y^-
-cx^+\dots)\; + (-1)^i\; \frac{2}{\sqrt{|d|}} \left(\frac{(4+f_{11}x^+)\lambda^k+\mu_2}{3}\right)^{3/2}(1+\dots),\;\;\\
\displaystyle \mbox{where}\;\; \mu_2 +  f_{11}x^+\lambda^k \geq -\frac{4}{3}\lambda^k\;\; \mbox{and}\;\; i=1,2, \;\;
\mbox{if}\;\; \nu = -1.  
\end{array}
\right.
\end{array}
\label{eq:criv+k}
\end{equation}

Analogous formulas can be obtained for the other bifurcation curves, given in~(\ref{eq:criv}),~(\ref{eq:ellfp}) for 
$\varphi =\pi/2, 2\pi/3$, and~(\ref{eq:twistless}).
This completes the proof of Theorem~\ref{thm:main}.

\section{On 1:4 resonance in the cubic H\'enon maps} \label{sec:14rez}

The structure of the strong 1:3 resonance, i.e. the structure of bifurcations related to the existence of fixed points 
with {eigenvalues}~$e^{\pm i 2\pi/3}$, in the case of  cubic maps~$\mathbf{C}_+^{\;1}$ and~$\mathbf{C}_-^{\;1}$ was studied 
in~\cite{DM00}.
In this section we analyze bifurcations related to the existence of fixed points with {eigenvalues}~$e^{\pm i \pi/2}$
(the main 1:4 resonance) in  maps~$\mathbf{C}_+^{\;1}$ and~$\mathbf{C}_-^{\;1}$.
We find conditions of nondegeneracy of the corresponding resonances
and give a description of accompanying bifurcations including {the main bifurcations in the degenerate cases.

Among the latter bifurcations we {pay} main attention to bifurcations of symmetric 
 4-periodic points, 
i.e. such 4-periodic orbits which has two points on the symmetry line~$R:\{x=y\}$. We will call such orbits~$R$-symmetric. 
Note that both maps~$\mathbf{C}_+^{\;1}$ and~$\mathbf{C}_-^{\;1}$ are reversible with respect to 
the involution~${\cal L}:\{x\to y,y\to x\}$, i.e. the maps preserve the form if 
one makes the coordinate change~${\cal L}$ and consider the inverse map. The line~$R$ is the line of 
fixed points of involution~${\cal L}$ and, thus,~$R$-symmetric periodic orbits can undergo
nondegenerate bifurcations of two types: symmetry preserving parabolic (hold) bifurcations and 
symmetry breaking pitch-fork bifurcations.  In the conservative case, there are two types of 
pitch-fork bifurcations: (i) when a symmetric elliptic fixed (periodic) point becomes a symmetric saddle fixed point 
and a pair of elliptic fixed points (symmetric each other with respect to~$R$) is born, and (ii) when a symmetric saddle 
fixed (periodic) point becomes a symmetric elliptic fixed point and a pair of symmetric each other saddle fixed points is born. 

We note that,
in the case under consideration, there are nonsymmetric 4-periodic orbits whose structure and bifurcations can be important 
to understand some thin details of 1:4 resonance. In this case, we study these orbits only numerically. 

\subsection{1:4 resonance in  map~$\mathbf{C}_+^{\;1}$.} \label{sec:14rez+}

The map~$\mathbf{C}_+^{\;1}$
has a unique fixed point with {eigenvalues}~$e^{\pm i \pi/2}$
at values of parameters~$M_1$ and~$M_2$ belonging to the curve
$$
L^+_{\pi/2}\;:\; M_1^2 = -\frac{4}{27}M_2 (M_2 -3)^2.
$$
Its equation is obtained from~(\ref{eq:ellfp}) with the sign ``$+$'' and $\cos\varphi=0$.

It was found in~\cite{MGon05} that the complex local normal form~(\ref{HeComplNew13}) near such a point has 
the following coefficients
$$
\displaystyle 8B_1(0) = 3 - 3 M_2,\;\;\;
8B_{03}(0)= 1 + 3 M_2
$$
and, thus, in~(\ref{D21D03})
$$
\displaystyle A   
= \frac{|3 - 3 M_2|}{|1 + 3 M_2|}.
$$
Since $M_2\leq 0$ in the case of curve~$L_{\pi/2}^+$, we always have 
$A>1$. Thus, in~(\ref{D21D03}) only condition $|B_{03}(0)|\neq 0$ is violated at~$M_2 = -1/3$.

The main local bifurcations which occur here at transition of the parameters
through curve~$L_{\pi/2}^+$ (except for the two points with~$M_2 = -1/3$) are illustrated in Figure~\ref{fig:res1_4chm1}. 
Curve~$L_{\pi/2}^+$  divides the~$(M_1,M_2)$-parameter plane into
two regions~I and~II. In both regions, the fixed point is generic elliptic, however, it is not generic for 
$(M_1,M_2)\in L_{\pi/2}^+$ although
the point is of elliptic type. The transition
from~I into~II is accompanied by the appearance of a resonant 1:4 chain containing 
the saddle and elliptic 4-periodic cycles around the fixed point.

\begin{figure}[tb]
\centering
\includegraphics[width=16cm]{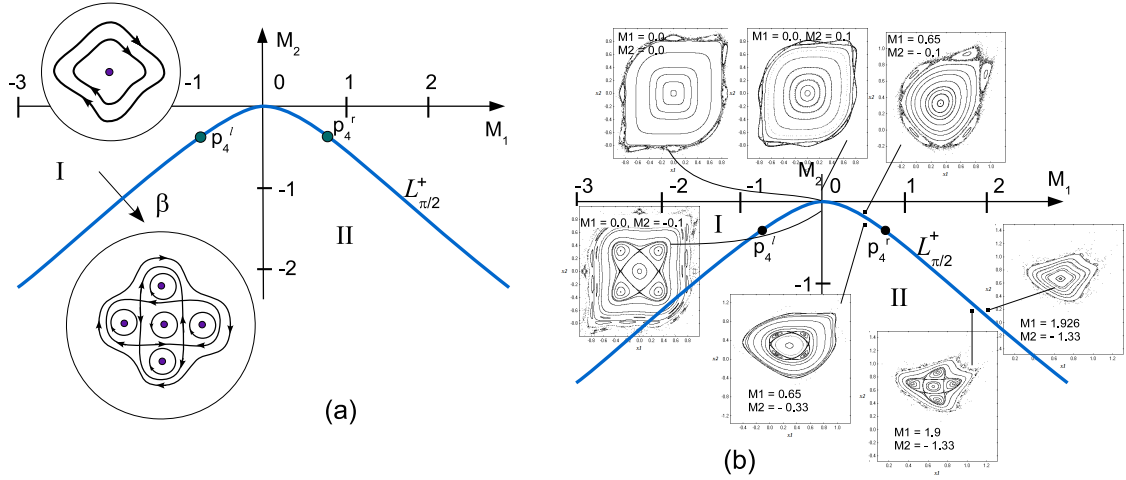}
\caption{{ Local bifurcations at transition
of the parameters through curve~$L_{\pi/2}^+$ for (a) the flow normal form~(\ref{HeComplNew13}); 
(b)  for~map $\mathbf{C}_+^{\;1}$,
the phase portraits are obtained numerically.}}
\label{fig:res1_4chm1}
\end{figure}

We note that in curve~$L_{\pi/2}^+$ there are two points~$p_4^l$ and~$p_4^r$ with~$M_2 = -1/3$, where~$B_{03} =0$.
Now we consider in more detail a neighborhood of the point~$p_4^r$ (point~$p_4^l$ can be studied in 
the same way). We find that~$p_4^r$ is an endpoint of two bifurcation curves~$B_4^r$ and~$\tilde B_4^r$.
These two curves divide domain~II into 3 domains,~II$_a$, II$_b$ and~II$_c$. In domain~II$_a$, the 1:4 chain looks as 
in Figure~\ref{pi2plusres} for $M_1 = 0.7$, $M_2=-0.5$. Thus, the chain contains the~$R$-symmetric elliptic 
4-periodic orbit (whose two points belong to the bisecting line~$x=y$). 
At crossing the line~$B_4^r$ (from~II$_a$ to~II$_b$), the elliptic orbit undergoes a pitch-fork bifurcation: 
it becomes a symmetric saddle 4-periodic orbit and a pair of symmetric each other elliptic 4-periodic orbits 
emerges. When the parameters change in  zone~II$_b$, some
bifurcations connected with the reconstruction of homo/heteroclinic structures take place, and, finally,
the saddle 4-periodic orbit of the chain which is not~$R$-symmetric  undergoes a
pitch-fork bifurcation at crossing curve~$\tilde B_4^r$ : two elliptic 4-periodic orbits merge into it and 
the orbit becomes elliptic. 
Thus, in domain~II$_c$, the 1:4 chain looks as in Figure~\ref{pi2plusres} for $M_1 = 0.75, M_2=-0.5$.

\begin{figure}[tb]
\centering
\includegraphics[width=14cm]{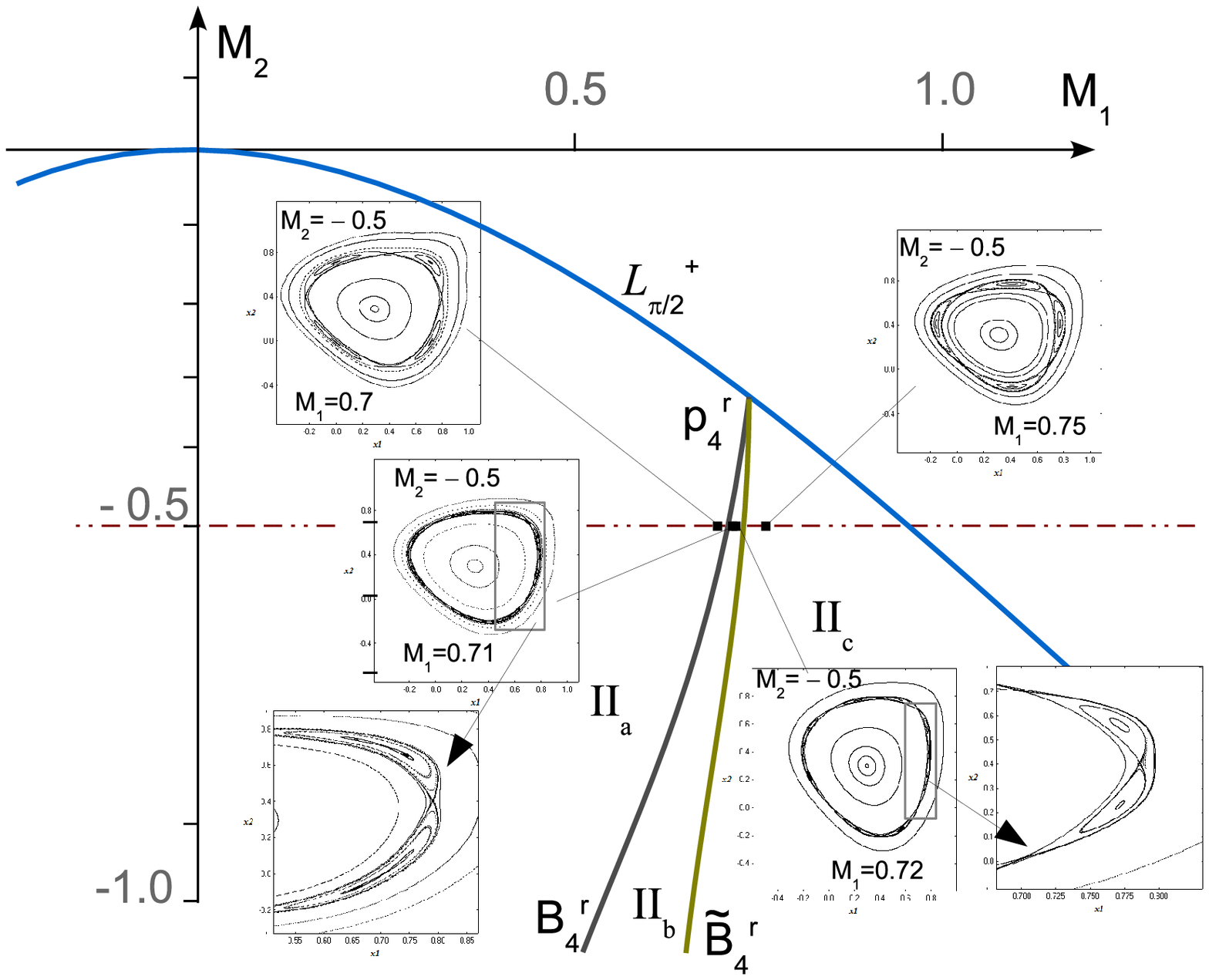}
\caption{{{ A fragment of bifurcation diagram at crossing curves~$B_4^r$ and~$\tilde B_4^r$ when~$M_1$ changes 
for a fixed~$M_2=-0.5$.
Effect of thinning of the 1:4 resonance chain is observed and some accompanied pitch-fork bifurcations 
with 4-periodic orbits are shown.
}}}
\label{pi2plusres}
\end{figure}

In Figure~\ref{res14plus} we interpret these bifurcations by means of the following flow normal form 
(which is invariant under rotation by angle $\pi/2$)
\begin{equation}
\displaystyle
\dot \zeta =  - 4 i \left( -\beta \zeta +
B_1 \zeta |\zeta|^2  +  \mu \zeta^{*3} +  \hat A \zeta^5 + B_2|\zeta|^4 \zeta + \hat C |\zeta|^2 \zeta^{*3} +  O(|\zeta|^7) \right), 
\label{nf1b30}
\end{equation}
where $\beta$ and $\mu$ are small real parameters, coefficients $B_1, B_2, \hat A$ and~$\hat C$ are real and $5 \hat A \equiv \hat C $ (in this case the divergence of~(\ref{nf1b30}) is zero,~\cite{Bir87}).
Note that system~(\ref{nf1b30}) is also reversible, it is invariant under the time reversal 
$t\to -t$ and the following involutions: ${\cal R}_1:\{x\to y, y\to x\}$, 
${\cal R}_2:\{x\to - y, y\to -x\}$, ${\cal R}_3:\{x\to -x, y\to y\}$ and 
${\cal R}_4: \{x\to x, y\to -y\}$. Accordingly, if~$B_1\beta<0$, the system has 8 nontrivial equilibria lying in pairs in the 
4 lines ${\cal R}_1,...,{\cal R}_4$ of fixed points of the corresponding involutions.    
\begin{figure}[tb]
\centering
\includegraphics[width=16cm]{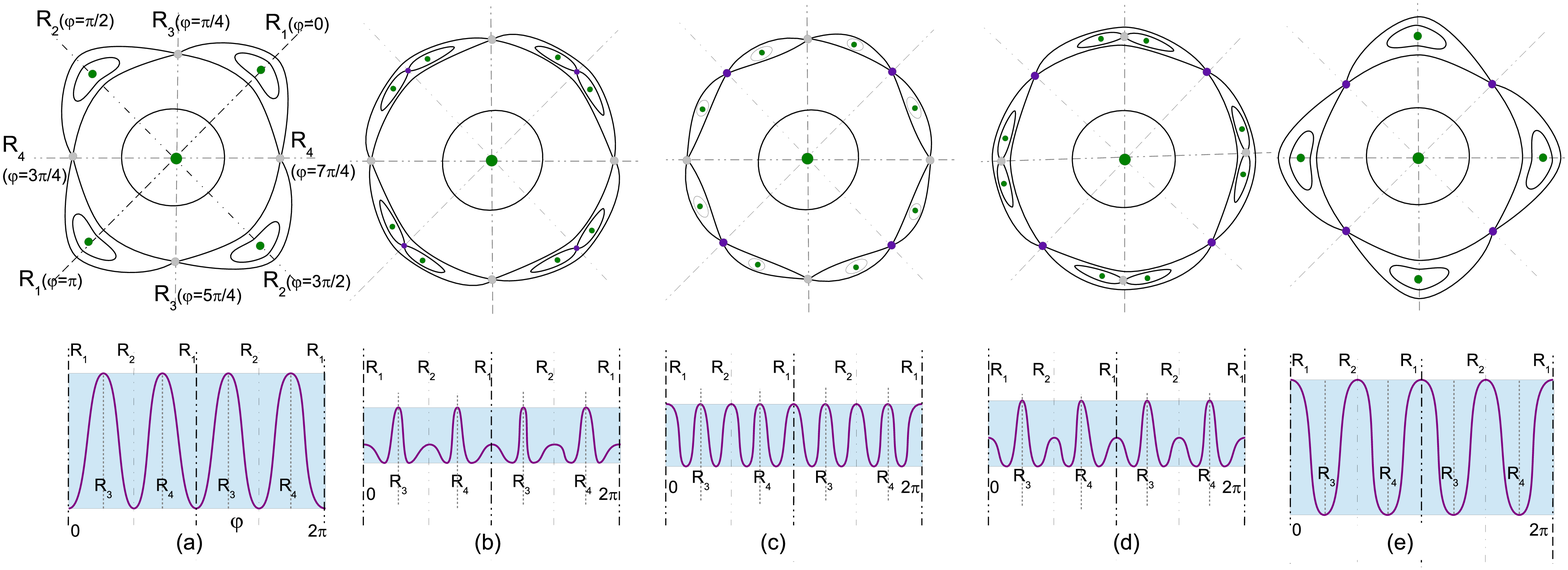}
\caption{{ The main steps of bifurcations inside the 1:4 resonance chain (upper plots), and the corresponding evolution of the appropriate cyclic
potential function (lower plots)
for the Hamiltonian flow~(\ref{nf1b30})  when~$\mu$ changes and~$B_1\beta<0$ is fixed, 
related to the cases
(a)  
$M_1=0.7$, (b) $M_1=0.71$, (c) a critical value $0.71<M_1<0.72$ corresponding to the homo/heteroclinic zone, (d) $M_1=0.72$, (e) $M_1=0.75$.
Here,~$R:\{x=y\}$ and~$R^*:\{x=-y\}$ are the lines of fixed points of the corresponding involutions. 
In the lower plots for the potential function, we need to identify the 
left and rights hand sides ($0$ and $2\pi$) of the bands.
}}
\label{res14plus}
\end{figure}

We see that the theoretical results, Figure~\ref{res14plus}, are in good accordance with the numerical evidence  
for a parameter path~$M_2 = -0.5$ in Figure~\ref{pi2plusres}.
Note that in Figure~\ref{res14plus} the case~(c) approximates the homo/heteroclinic zone of domain~II$_b$, where the invariant manifolds of the saddles are reconstructed
from the~(b)-configuration to the~(d)-configuration. Simultaneously when the parameters change in 
domain~II$_b$, we observe the effect of thinning of the 1:4 resonance chain. This is connected with the fact that 
the normal form~(\ref{nf1b30}) with~$\mu=0$ and~$C=0$ 
corresponds to a nonlinear center.

\subsection{1:4 resonance in map~$\mathbf{C}_-^{\;1}$ } \label{sec:14rez-}

The cubic map~$\mathbf{C}_-^{\;1}$ has a fixed point with {eigenvalues}~$e^{\pm i \pi/2}$ for
$(M_1,M_2) \in L_{\pi/2}^-$, where curve~$L_{\pi/2}^-$ has equation
\begin{equation}
L^-_{\pi/2}\;:\; M_1^2 = \frac{4}{27}M_2 (M_2 -3)^2.
\label{eq:res-}
\end{equation}
This equation is obtained from~(\ref{eq:ellfp}) with the sign ``$-$'' and $\cos\varphi=0$.

Note that curve~$L_{\pi/2}^-$ has a self-intersection point ($M_1=0,M_2=3$),
and only in this
moment the map has simultaneously two fixed points
with {eigenvalues}~$e^{\pm i \pi/2}$.
The coefficients of the local complex normal form~(\ref{HeComplNew13})
are,~\cite{MGon05},
$$
8 B_1 = -3 + 3 M_2,\;
8 B_{03}= -1 - 3 M_2.
$$
Since $M_2\geq 0$ in curve~$L_{\pi/2}^-$, in~(\ref{D21D03}) we have   
$B_{03} <0$ and 
$$
\displaystyle A =  \frac{|3 - 3 M_2|}{|1 + 3 M_2|}.
$$
\begin{figure}[tb]
\centering
\includegraphics[height=8cm]{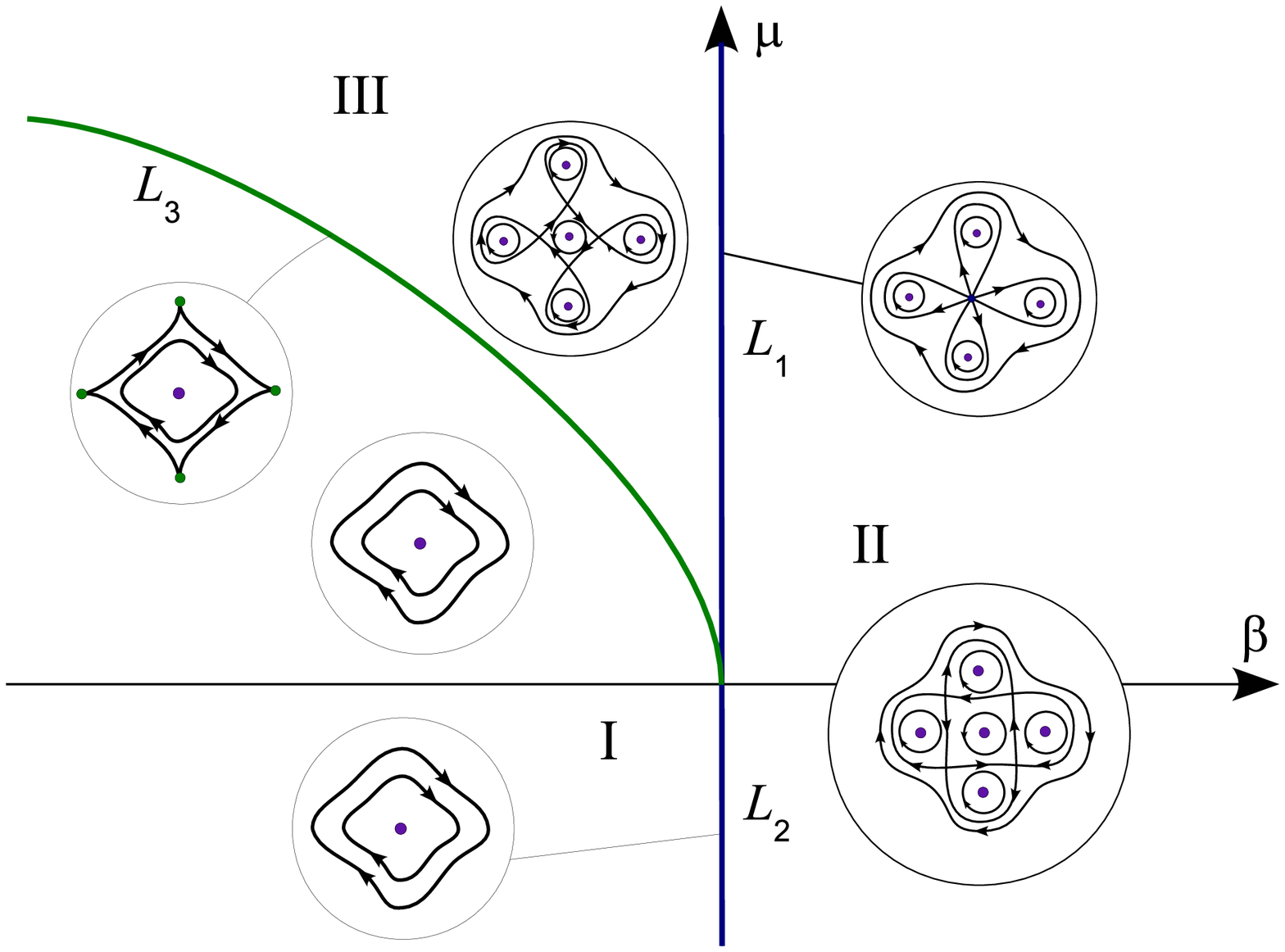}
\caption{{ Main elements of bifurcation diagram in the case $A=1$.}}
\label{fig_pi2-A1}
\end{figure}
Thus, both  cases~$A>1$ (when $\displaystyle 0\leq M_2 < {1}/{3}$) and~$A<1$ (when $\displaystyle M_2 > {1}/{3}$) are possible here. In~$L_{\pi/2}^-$, there are two points 
$P^+(M_1 = 16/27, M_2 = {1}/{3})$  and $P^-(M_1 = -16/27, M_2 = {1}/{3})$
where the case~$A=1$ takes place.
Note that if~$A<1$, the fixed point $O$ is a saddle with 8 separatrices for
 $(M_1,M_2)\in L_{\pi/2}^-$,  and the main bifurcations are connected with the reconstruction of a saddle 4-periodic 
 orbit (as at the transition between the domains~II and~III from the bifurcation diagram of Figure~\ref{fig_pi2-A1}).

In the degenerate case~$A=1$, the main bifurcations can
be described by means of a two parameter family of the Hamiltonian system
\begin{equation}
\displaystyle
\dot \zeta =  i \beta \zeta +
i(1+\mu) \zeta |\zeta|^2 + i \zeta^{*3}  + i B_2|\zeta|^4\zeta   +  i C |\zeta|^2 \zeta^{*3} +  O(|\zeta|^7),
\label{2p5_2}
\end{equation}
where  $\beta$ and $\mu$ are parameters, $B_2$ and $C$ are real coefficients, see more details in~\cite{Bir87,Gelf,SV09}.
The typical bifurcation diagram\footnote{Note that the nondegeneracy conditions in this case include also 
the inequality $B_2 \neq C$,~\cite{Bir87}.}
for family~(\ref{2p5_2}) is shown in Figure~\ref{fig_pi2-A1}.
It contains three bifurcation curves which divide the~$(\beta,\mu)$-parameter plane into 3 domains. Crossing the curves
$L_1:\{\beta=0,\mu>0\}$ and  $L_2:\{\beta=0,\mu<0\}$ corresponds to the reconstruction of nonzero saddle equilibria. Curve~$L_3$ corresponds to the appearance of 4 nonzero parabolic equilibria (this bifurcation is nondegenerate for 
the family~(\ref{2p5_2}) since it is invariant under rotation by angle~$\pi/2$). Note that the trivial equilibrium is a
nondegenerate conservative center for~$\beta \neq 0$; it is a degenerate conservative center for
$(\beta,\mu)\in L_2$ and a saddle with 8 separatrices for~$(\beta,\mu)\in L_1$.}
\begin{figure}[tb]
\centering
\includegraphics[height=14cm]{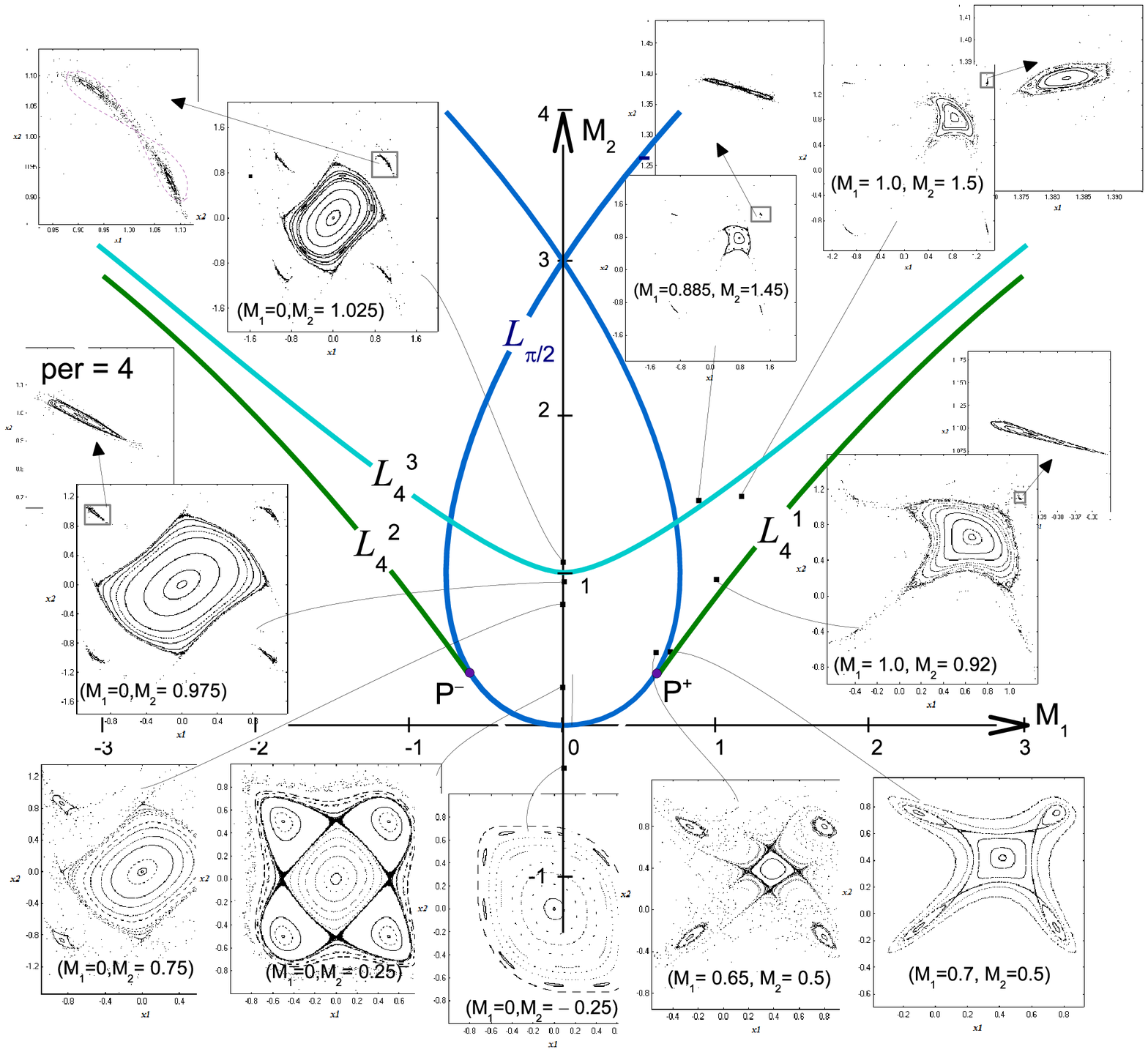}
%
\caption{{ Some elements of the bifurcation diagram for 1:4 resonance in the case of map~$\mathbf{C}_-^{\;1}$.}}
\label{fig:res1_4_chm2}
\end{figure}

Concerning map~$\mathbf{C}_-^{\;1}$,  the main elements of the bifurcation diagram related to 1:4 resonance are shown 
in Figure~\ref{fig:res1_4_chm2}.  Here
the 4 main bifurcation curves,~$L_{\pi/2}^-$, $L_4^1, L_4^2$ and~$L_4^3$, are shown in the~$(M_1,M_2)$-parameter plane. 
The equation for~$L_{\pi/2}^-$ is given by~(\ref{eq:res-}), the remaining curves have the following equations
\begin{equation}
\label{p4-0}
\begin{array}{l}
\displaystyle  L_4^{1,2} :\;     M_1 = \pm \frac{2}{3 \sqrt{3}} (1 + M_2)^{3/2},\;  M_2 > 1/3,  \\
\displaystyle  L_4^3 :\;    M_1 = \pm\frac{2}{3 \sqrt{3}} (2 + M_2) \sqrt{M_2 - 1},  \\
\end{array}
\end{equation}

Curves~(\ref{p4-0}) correspond to nondegenerate parabolic bifurcations for~$R$-symmetric 4-periodic orbits. 
The birth of a pair of saddle and elliptic 4-periodic orbits occurs at upward crossing these curves.
Curves~$L_4^1$ and~$L_4^2$ are (quadratically) tangent to curve~$L_{\pi/2}^-$ at the points~$P^\pm$,
see Figure~\ref{fig:res1_4_chm2}. In principle, bifurcations of map~$\mathbf{C}_-^{\;1}$, when~$M_1$ and~$M_2$ vary 
near~$P^\pm$, look similar to the bifurcations in the flow normal form~(\ref{2p5_2}), compare Figures~\ref{fig_pi2-A1} 
and~\ref{fig:res1_4_chm2} (in the latter figure, such bifurcations are displayed for the case of point~$P^+$). 
The presence of curve~$L_4^3$
shows
that the structure of 1:4 resonance in the case of map~$\mathbf{C}_-^{\;1}$ is not trivial that can be imagined 
from Figure~\ref{fig:res1_4_chm2}, where the main elements of the corresponding bifurcation diagram are shown.

\end{document}